\documentclass[a4paper]{extarticle}
\usepackage[english]{babel}
\usepackage{graphicx}
\usepackage{framed}
\usepackage[normalem]{ulem}
\usepackage{amsmath}
\usepackage{stmaryrd}
\usepackage{comment}
\usepackage{amsthm}
\usepackage{amssymb}
\usepackage{amsfonts}
\usepackage{enumerate}
\usepackage{xcolor}
\usepackage{bbm}
\usepackage{enumitem}
\usepackage[utf8]{inputenc}
\usepackage[T1]{fontenc}
\usepackage[top=1 in,bottom=1in, left=1 in, right=1 in]{geometry}
\usepackage{hyperref}
\usepackage{authblk}

\usepackage{ amsmath, amsfonts, amssymb}

\newtheorem{theorem}{Theorem}
\newtheorem{lemma}{Lemma}

\newtheorem{proposition}{Proposition}

\newtheorem{definition}{Definition}
\newtheorem{corollary}{Corollary}
\newtheorem{remark}{Remark}

\newcommand{\cA}{\ensuremath{\mathcal A}}
\newcommand{\cB}{\ensuremath{\mathcal B}}
\newcommand{\cC}{\ensuremath{\mathcal C}}
\newcommand{\cD}{\ensuremath{\mathcal D}}
\newcommand{\cE}{\ensuremath{\mathcal E}}
\newcommand{\cF}{\ensuremath{\mathcal F}}
\newcommand{\cG}{\ensuremath{\mathcal G}}
\newcommand{\cH}{\ensuremath{\mathcal H}}

\newcommand{\cM}{\ensuremath{\mathcal M}}
\newcommand{\cN}{\ensuremath{\mathcal N}}

\newcommand{\cP}{\ensuremath{\mathcal P}}

\newcommand{\cS}{\ensuremath{\mathcal S}}

\newcommand{\bbE}{{\ensuremath{\mathbb E}} }

\newcommand{\bbN}{{\ensuremath{\mathbb N}} }

\newcommand{\bbR}{{\ensuremath{\mathbb R}} }
\newcommand{\bbS}{{\ensuremath{\mathbb S}} }

\newcommand{\N}{\mathbb{N}}

\newcommand{\be}{\begin{equation}}
\newcommand{\ee}{\end{equation}}
\newcommand{\beq}{\begin{eqnarray}}
\newcommand{\eeq}{\end{eqnarray}}

\newcommand{\1}{{1} \hspace{-0.25 em}{\rm I}}

\newcommand{\Bt}{\tilde{B}}
\newcommand{\Bf}{\overleftarrow{dB}}
\newcommand{\bal}{\bar{\alpha}}
\newcommand{\bC}{\bar{C}}

\newcommand{\ut}{\tilde{u}}

\newcommand{\R}{\mathbb{R}}

\newcommand{\ced}{\end{proof}}

\newcommand{\hF}{\widehat{\cF}}
\newcommand{\hG}{\widehat{\cG}}
\newcommand{\hP}{\widehat{P}}
\newcommand{\hcP}{\widehat{\cP}}
\newcommand{\hE}{\widehat{\mathbb{E}}}

\newcommand{\hL}{\widehat{L}}
\newcommand{\hM}{\widehat{M}}
\newcommand{\hY}{\hat{Y}}
\newcommand{\hZ}{\hat{Z}}
\newcommand{\hS}{\widehat{S}}

\usepackage{lipsum}
\usepackage{amsfonts}
\usepackage{graphicx}
\usepackage{epstopdf}
\usepackage{algorithmic}
\ifpdf
  \DeclareGraphicsExtensions{.eps,.pdf,.png,.jpg}
\else
  \DeclareGraphicsExtensions{.eps}
\fi

\title{Stochastic PDEs driven by $G-$Brownian motion and the associated Backward Doubly Stochastic Differential Equations\thanks{The work of the first author benefited from the support of the ANR project "Efficient inference for large and high-frequency data" (ANR-21-CE40-0021).\\
The work of the second author is supported by the National Natural Science Foundation of China (12326603, 12271103, 12031009), National Key R\&D Program of China (2022YFA1006101) and Shanghai Science and Technology Commission Grant (21ZR140860).
}}
\author[1]{Laurent Denis \thanks{Laurent.Denis@univ-lemans.fr}}
\author[2]{Jing Zhang\thanks{zhang\_jing@fudan.edu.cn}}
\affil[1]{\small Laboratoire Manceau de Math\'ematiques, Institut du Risque et de l'Assurance, Le Mans Universit\'e.  }
\affil[2]{School of Mathematical Sciences, 
Fudan University}

\begin{document}

\date{}
\maketitle

\begin{abstract}
Our aim is to study the well-posedness of quasilinear stochastic partial differential equations driven by $G-$Brownian motion (GSPDEs for short) and the associated backward doubly stochastic differential equations (GBDSDEs for short). We first prove the existence and uniqueness of weak solution to GSPDEs by analytical approach, and then solve the corresponding GBDSDEs. Finally, the relation between GSPDEs and GBDSDEs is established.
\end{abstract}

\tableofcontents

\section{Introduction}
In this paper, we consider the following stochastic partial differential equations  driven by $G$-Brownian motion (GSPDEs for short) on $\mathbb{R}^d$,
\begin{equation}\label{GSPDE}
\begin{split}
du_t (x) + \left[\partial_i(a^{i,j}(x)\partial_j u_t (x))+ f_t(x,u_t (x),\nabla u_t (x))\right]&dt\\+g_t(x,u_t(x),\nabla u_t(x))&\cdot \overleftarrow{dB}_t = 0,
\end{split}
\end{equation}
over the time interval $[0,T]$, with a given final  deterministic condition $u_T = \Psi$ and $f, 
g := \big(g_1, \cdots,g_{l}\big)$ non-linear random functions. The backward stochastic integral is defined w.r.t. an $l-$dimensional $G-$Brownian motion $B$, which will be defined later. Our aim is to study the well-posedness of \eqref{GSPDE} and give  a doubly stochastic interpretation of the solution in terms of a forward integral with respect to $M$,  the martingale part of the Hunt process associated to the second order operator and  backward integral with respect to  $B$. \\
In the classical framework, if $B$ is a standard Brownian motion, stochastic partial differential equations (SPDEs for short) of form \eqref{GSPDE} have been extensively studied, including the well-posedness, $L^p-$estimates, maximum principle, by analytical approach, for the details, we refer to \cite{Denis, DenisStoica, DMS05, DMS09, DM11}. 
Pardoux and Peng \cite{PardouxPeng94} first introduced backward doubly stochastic differential equations (BDSDEs for short) and established the relation between a class of quasilinear SPDEs and BDSDEs, in which the classical solutions were considered and matrix $a$ was assumed to be smooth enough. Then in \cite{BM01}, Bally and Matoussi gave the probabilistic representation of Sobolev solutions to parabolic semilinear SPDEs in terms of BDSDEs by using the stochastic flow approach. \\
If matrix $a$ is only assumed to be measurable and associated to a Dirichlet form, Bally, Pardoux and Stoïca in \cite{BPS} (see also \cite{Stoica}), following the works of Fukushima \cite{FOT}, established the link between semilinear partial differential equations with terminal conditions and BSDEs driven by the Hunt process associated to the Dirichlet form.\\ 
Motivated by uncertainty problems, risk measures and the superhedging in finance, Peng \cite{Peng2007,Peng2008,Peng2010} introduced 
$G-$Brownian motion. The expectation $\bbE[\cdot]$ associated with $G-$Brownian motion is a sublinear expectation called $G-$expectation. 
The stochastic calculus with respect to the $G-$Brownian motion has been established in \cite{Peng2010}. Then based on the stochastic calculus under $G$-expectation framework and a way parallel to the classical SDEs theory, the well-posedness of stochastic differential equations driven by $G-$Brownian motion (GSDEs for short) are obtained, see for example \cite{BaiLin, Gao, LinY,Peng2010}. { Hu et al. \cite{HuJiPengSong} proved the existence and uniqueness of the solution $(Y,Z,K)$, with $K$ a decreasing $G$-martingale, to backward stochastic differential equations driven by $G$-Brownian motion (GBSDEs for short). They applied the partition of unity theorem to construct a new type of Galerkin approximation instead of the well-known Picard iteration. The readers are also refereed to \cite{Peng2010, LiuG} and the references therein on GBSDEs. As we know, one of the motivations of BSDEs is to give the probabilistic interpretation of nonlinear PDEs, which is useful in applications and numerical methods. For GBSDEs, the similar results also exist. Hu et al. \cite{HuJiPengSongSPA} showed that the GBSDEs in Markovian case corresponds to a fully nonlinear PDEs. Peng and Song \cite{PengSong} established the correspondence between BSDEs and a type of quasilinear path-dependent PDEs in the corresponding $G$-Sobolev spaces.} \\
Our aim is to study the link between GSPDEs and GBDSDEs. More precisely, if we denote by $X$ the Hunt process whose generator is $L=\sum_{i,j=1}^d \partial_i (a^{i,j}\partial_j )$ and by $M$ its martingale part and if $\sigma (x)=(a(x))^{1/2}$ then under standard assumptions, if $u$ is the solution of \eqref{GSPDE}, the pair of processes 
 $$Y_t =u(t,X_t),\ \ Z_t =\nabla u(t,X_t)$$
 is solution of the following GBDSDE:
 \begin{equation*}\
Y_t=\xi +\int_t^T f(s,Y_s,Z_s\sigma(X_s))ds +\int_t^T g(s,Y_s ,Z_s\sigma(X_s))\cdot\Bf_s-\int_t^T Z_s\cdot dM_s.
\end{equation*}
Let us point out that, even in the standard case where $B$ is a standard Brownian motion, this representation is, to the best of our knowledge, new.\\
Finally, we make the following remarks. In this work, we consider an SPDE involving a linear differential operator and nonlinear noise, since in our setting the associated GBDSDEs are driven by a $G$-Brownian motion and a standard centred martingale that are independent under the product probability measures. An interesting question is to investigate the case involving a nonlinear second-order differential operator and a standard Brownian motion—thus yielding linear noise. This situation has been studied in the framework of second-order backward stochastic differential equations (2BSDEs) by Matoussi et al. \cite{MPS19}. One may also analyse this problem within the $G$-expectation framework.
The most general case involves nonlinearities appearing simultaneously in both the operator and the noise. This leads to an intrinsic mathematical challenge: it is impossible to construct two independent $G$-Brownian motions under a nonlinear expectation. Therefore, identifying an appropriate notion of dependence between two $G$-Brownian motions becomes the key to addressing the problem. We leave these two cases for future work. \\
 The breakdown of this paper is the following. In the second section, we introduce the backward stochastic integral with respect to the $G-$Brownian motion. Section 3 is devoted to proving the existence and uniqueness of weak solution to the GSPDE \eqref{GSPDE} by analytical approach. Then in Section 4, we study the well-posedness of the associated backward doubly stochastic differential equations. Finally, the relation between GSPDEs and GBDSDEs is established in the fifth section.  

\section{Framework}

 
\subsection{Hypotheses on process $B$}\label{DefCap}
Let $T>0$, $l\in\bbN^*$ be fixed. We consider  
$\Omega^B=C([0,T];\bbR^l)$ 
the space of $\bbR^l$-valued continuous functions, $\omega$, on $[0,T]$ and denote by $B$ 
the coordinates process on $\Omega^B$ and the forward and backward filtrations:
$$\forall t\in [0,T],\ \cF^B_t =\sigma\{ B_s ;\ s\leq t\}\makebox{ and } \cF^{B}_{t,T} =\sigma\{ B_s -B_t ;\ t\leq s\leq T\}.$$
We introduce the ``backward processes'': $\forall t\in [0,T],\Bt_t =B_{T-t}-B_{T}$ and $\overleftarrow{B}_t=B_{T-t}$.\\ 
Consider $G(\cdot):\bbS^l\rightarrow\bbR$ a monotonic and sublinear function, where $\bbS^l$ denotes the space of $l\times l$ symmetric matrices.
By Theorem 2.1 in \cite{Peng2010}, we know that there exists a bounded, convex and closed subset 
$\Theta\subset\bbS^l$ such that $G(A)=\frac{1}{2}\sup_{\beta\in\Theta}tr[\beta\beta^TA],\  A\in\bbS^l$.
Let $P_0^B$ be the standard Wiener measure on $(\Omega^B ,(\cF^B_t )_{t\in [0,T]})$ and $\cA^\Theta$ be the collection of all $\Theta-$valued $(\cF^{B}_{t,T})_{t\in [0,T]}-$adapted processes on the interval $[0,T]$. For each fixed 
$\theta\in\cA^\Theta$, let $P_\theta$ be the law of the process $(\int_t^T\theta_s \Bf_s)_{t\in [0,T]}$ under the Wiener measure $P^B_0$, where $\Bf$ denotes the backward integral. 
Clearly, under each $P_\theta$, $(\Bt_t )_{t\in [0,T]}$ is a centered martingale. 

\begin{remark}\label{Rqb}
   Since $\Theta$ is bounded, there exists a constant $\bar{\sigma}$ such that for all $\theta\in\cA^\Theta$, under  $P_\theta$ the bracket of $\Bt$ satisfies:
   $$\left( d\langle \Bt^i ,\Bt^j\rangle_t\right)_{1\leq i,j\leq l} \leq \bar{\sigma}^2 I_d dt,\quad P_\theta-\makebox{a.s.}$$
      where $I_d $ is the $l\times l$ identity matrix and the inequality is in the sense of positive definite symmetric matrices.
   \end{remark}
We denote by $\cP_1=\{P_\theta:\ \theta\in\cA^\Theta\}$ and $\cP=\bar{\cP_1}$ the closure of $\cP_1$ under the topology of weak convergence.
Proposition 49 in \cite{DHP} ensures that $\cP_1$ is tight hence $\cP$ is weakly compact.  We set
$$c(A):=\sup_{P\in\cP}P(A),\ \ A\in\cB(\Omega^B).$$
From Theorem 1 of \cite{DHP}, we know that $c$ is a Choquet capacity.
We can then introduce the notion of ``quasi-sure" (q.s.).
\begin{definition}
    A borelian set $A\subset\Omega^B$ is called polar if $c(A)=0$. A property is said to hold ``quasi-surely" (q.s.) if it holds outside a polar set.
   \end{definition}
\begin{remark}
In other words, $A\in\cB(\Omega^B )$ is polar if and only if $P(A)=0$ for any $P\in\cP$.
     \end{remark}

\begin{definition}
A mapping $\xi$ on $\Omega^B$ with values in a topological space is said to be quasi-continuous (q.c.) if 
$\forall\epsilon>0$, there exists an open set $O$ with $c(O)<\epsilon$ such that $\xi|_{O^c}$ is continuous. 
\end{definition}

We now introduce the $G$-expectation. For each $\xi\in L^0(\Omega^B)$ (the space of all $\cB(\Omega^B)$-measurable real functions on $\Omega^B$) such that $E_P[\xi]$ exists for each $P\in\cP$, where $E_P$ denotes the expectation under probability $P$,  we set
$${\bbE}[\xi]:=\sup_{P\in\cP}E_P[\xi].$$  
Following \cite{Peng2010},   $(\Bt_t)_{t\in [0,T]}$ is a $G$-Brownian motion under the non-linear expectation $\bbE$.
\begin{remark} $(B_t )_{\in [0,T]}$ is not necessarily a $G$-Brownian motion.
\end{remark}
Let $H$ be a separable Hilbert space equipped with the norm
$\left\|\cdot\right\|_H$ and the scalar product $(\cdot,\cdot)_H$. Since it is separable, we can consider $(e_j)_{j\geq1}$ an orthonormal basis of $H$ and we consider that the dimension of $H$ is infinite (if it is finite, all the results are, of course, still true).
\begin{equation*}\begin{split}L_{ip}(\Omega^B;H):=&\Big\{\sum_{k=1}^N\varphi_k(\Bt_{t_1},...,\Bt_{t_n})e_k:\forall n\geq1,\ t_1,...,t_n\in[0,T],\\&
\qquad N\geq 1, \varphi_k\in C_{b,Lip}(\bbR^{l\times n}),\forall k\in\{1,...,N\}\Big\}.\end{split}\end{equation*}
We denote by $L^p_G(\Omega^B;H)$  the topological completion of $L_{ip}(\Omega^B;H)$  
w.r.t.  the norm $\big[\bbE [\|\cdot\|_H^p ]\big]^{\frac{1}{p}}$ for $1\leq p< \infty$. 
\begin{proposition} \label{quasicontinousversion}
Each element in $L^p_G(\Omega^B;H)$ has a quasi-continuous version. 
\end{proposition}
\begin{proof} We can do a similar argument as in Proposition 24 in \cite{DHP}. 
\end{proof}
\begin{remark} (i) If $H=\R$ we omit it in the notations.\\
(ii) From Theorem 52 in \cite{DHP}, we have another characterization of $L^1_G(\Omega^B)$:
$$L^1_G(\Omega^B)=\Big\{\xi\in L^0(\Omega^B):\ \xi\ has\ a\ q.c.\ version,\ \lim_{n\rightarrow\infty}{\bbE}[|\xi|\1_{|\xi|>n}]=0\Big\}.$$
(iii) Rigorously speaking,   an element of the Banach space  $L^p_G(\Omega^B;H)$ is a class of equivalence of functions but, as usual, we do
not take care about the distinction between  classes and their representatives which are equal quasi-surely.
\end{remark}

 \subsection{Backward stochastic integrals w.r.t. $G$-Brownian motion}{\label{G-integral}}
We now construct the backward stochastic integral with respect to $G$-Brownian motion. Let $H$ be a given separable Hilbert space and denote by 
\begin{itemize}
    \item $M_G^{2,0}([0,T], (\cF^B_{t,T})_{t\in [0,T]};H)$, the set  of simple processes $\Xi$ such that  for a given partition $\pi_T=\{t_0,...,t_N\}$ of $[0,T]$,
$$\Xi_t(\omega):=\sum_{i=0}^{N-1}\xi_{i+1}(\omega)I_{(t_i,t_{i+1}]}(t),$$where
for each $i=0,1,...,N-1$, $\xi_{i+1}\in L^2_G(\Omega^B;H)$ and is $\mathcal{F}^B_{t_{i+1},T}$-measurable;
\item $M^2_G([0,T],(\cF^B_{t,T})_{t\in [0,T]};H)$, the completion of $M_G^{2,0}([0,T], (\cF^B_{t,T})_{t\in [0,T]};H)$ \\under the
norm$$\left\|\Xi\right\|_{M^2_G([0,T],(\cF^B_{t,T})_{t\in [0,T]};H)}:=\bigg\{\bbE\Big[\int_0^T\left\|\Xi_t\right\|_H^2dt\Big]\bigg\}^{1/2};$$
\item $S^2_G ([0,T],(\cF^B_{t,T})_{t\in [0,T]};H)$, the closed set of continuous processes in \\$M^2_G([0,T],(\cF^B_{t,T})_{t\in [0,T]};H)$ such that 
$$\bbE\Big[\sup_{t\in [0,T]}\|\Xi_t\|_H^2\Big]<+\infty.$$
\end{itemize}
Assume first that $\Xi:=(\Xi^1,\cdots,\Xi^l)$ belongs to $(M^{2,0}_G([0,T],(\cF^B_{t,T})_{t\in [0,T]};H))^l$ of the form 
$$\forall j\in\{1,\cdots,l\},\ \forall t\in[0,T],\ \Xi_t^j(\omega):=\sum_{i=0}^{N-1}\xi_{i+1}^j(\omega)I_{(t_i,t_{i+1}]}(t),$$ 
we define for any $t\in[0,T]$ :
$$\overleftarrow{I}^\Xi_t :=\int_t^T \Xi_s\cdot \Bf_s=\int_0^{T-t}\Xi_{T-s}\cdot d\Bt_s:=\sum_{j=1}^l\sum_{i=0}^{N-1}\xi_{i+1}^j(B^j_{t_i\vee t}-B^j_{t_{i+1}\vee t}).$$

\begin{proposition}\label{propsi}
The mapping $$\overleftarrow{I}:(M^{2,0}_G([0,T],(\cF^B_{t,T})_{t\in [0,T]};H))^l\rightarrow S^2_G ([0,T],(\cF^B_{t,T})_{t\in [0,T]};H)$$ defined by $\overleftarrow{I} (\Xi )=\left( \overleftarrow{I}^\Xi_t \right)_{t\in [0,T]}$ is a continuous linear mapping and can therefore be extended to a continuous linear operator from $(M^{2}_G([0,T],(\cF^B_{t,T})_{t\in [0,T]};H))^l$ into $S^2_G ([0,T],(\cF^B_{t,T})_{t\in [0,T]};H)$ that we still denote by $\overleftarrow{I}$. Moreover, the following properties hold for all  $\Xi\in (M^{2}_G([0,T],(\cF^B_{t,T})_{t\in [0,T]};H))^l$: 
\begin{enumerate}
\item $\forall t\in[0,T]$, $\bbE\big[\int_t^T\Xi_s\cdot\Bf_s\big]=0$;
\item $\forall t\in[0,T]$, $\bbE\big[\big\|\overleftarrow{I}^\Xi_t\big\|^2_H\big]\leq\bar{\sigma}^2\bbE\left[\int_t^T\left\|\Xi_s\right\|_{H^l}^2ds\right]$, where $\bar{\sigma}$ is defined in Remark \ref{Rqb};
\item $\overleftarrow{I}$ satisfies the Doob's inequality:
\begin{equation}\label{Doob}
\bbE\Big[\sup_{t\in[0,T]}\big\|\overleftarrow{I}^\Xi_t\big\|^2_{H}\Big]\leq
4\bar{\sigma}^2\bbE\int_0^T\left\|\Xi_t\right\|^2_{H^l}dt.
\end{equation}
\end{enumerate}
\end{proposition}
\begin{proof}
Under each $P\in\cP$, it is clear that  $E_{P}\big[\int_t^T\Xi_s\cdot\Bf_s\big]=0$, thus we have $\bbE\big[\int_t^T\Xi_s\cdot\Bf_s\big]=0$. 

By a direct calculation, we obtain 
\begin{equation*}\begin{split}
\bbE\big[\big\|\overleftarrow{I}^\Xi_t\big\|^2_H\big]&=\sup_{P\in\cP}E_{P}\big[\|\overleftarrow{I}_t^\Xi\|_H^2\big]\\&=\sup_{P\in\cP}E_{P}\bigg[\sum_{i,j=1}^l\int_0^{T-t}(\Xi^i_{T-s},\Xi^j_{T-s})d\langle \tilde{B}^i,\tilde{B}^j\rangle_s\bigg]
\\&\leq\bar{\sigma}^2\sup_{P\in\cP}E_{P}\bigg[\int_0^{T-t}\|\Xi_{T-s}\|^2_{H^l}ds\bigg]\\&=\bar{\sigma}^2\bbE\bigg[\int_t^{T}\|\Xi_{s}\|^2_{H^l}ds\bigg].
\end{split}\end{equation*}
The Doob's inequality under each $P\in\cP$ yields
\begin{equation*}\begin{split}
\bbE\bigg[\sup_{t\in[0,T]}\left\|\overleftarrow{I}_t^\Xi\right\|_H^2\bigg]&=\sup_{P\in\cP}E_{P}\bigg[\sup_{t\in[0,T]}\left\|\overleftarrow{I}_t^\Xi\right\|_H^2\bigg]
\\&\leq\sup_{P\in\cP}4E_{P}\bigg[\sum_{i,j=1}^l\int_0^T(\Xi^i_{T-s},\Xi^j_{T-s})d\langle \tilde{B}^i,\tilde{B}^j\rangle_s\bigg]
\\&\leq4\bar{\sigma}^2\sup_{P\in\cP}E_{P}\bigg[\int_0^T\left\|\Xi_{T-s}\right\|^2_{H^l}ds\bigg]\\&=4\bar{\sigma}^2\bbE\bigg[\int_0^T\left\|\Xi_{s}\right\|^2_{H^l}ds\bigg], 
\end{split}\end{equation*}
which ends the proof.
\end{proof}

\section{The well-posedness of GSPDEs}
In this section, we study the well-\\posedness of GSPDE \eqref{GSPDE}. 
Let $d\in\N^*$ and $L^2(\mathbb{{R}}^d) $ be the space of square integrable functions with respect to Lebesgue measure on $\mathbb{R}^d$, which is an Hilbert space equipped with the  inner product and norm 
$$(u,v):=\int_{\mathbb{R}^d}u(x)v(x)dx,\quad \|u\|:=\left(\int_{\mathbb{R}^d}u^2(x) dx\right)^{\frac 12}, \ \  u,v\in L^2(\mathbb{{R}}^d). $$ 
Let  $H^1(\mathbb{R }^d):=\{u:u\in L^2(\mathbb{R}^d), |\nabla u|\in  L^2(\mathbb{R}^d )\}$  be the first order Sobolev space equipped with inner product and norm 
$$\left( u,v\right) _{H^1\left( {\mathbb{R}^d}\right) }:=\left( u,v\right) +\left( \nabla u, \nabla v \right),
\quad \left\| u\right\|_{H^1\left({\mathbb{R}^d}\right) }:=\left( \left\| u\right\|^2+\left\| \nabla u\right\|^2\right) ^{\frac 12},$$
where $\nabla u (x):= (\partial_1 u (x), \cdots, \partial_du(x))$ is  the gradient of  function $u$.\\
Denote $L:=\sum_{i,j=1}^d\partial_i(a^{i,j}\partial_j)$ which is a symmetric second order differential operator defined on $\R^d$, with domain $\mathcal{D}(L)$. We assume that
$a(x):=(a^{i,j}(x))_{i,j}$ is a measurable symmetric matrix defined on
$\R^d$ which satisfies the uniform ellipticity
condition$$\lambda|\xi|^2\leq\sum_{i,j=1}^d
a^{i,j}(x)\xi^i\xi^j\leq\Lambda|\xi|^2,\ \forall x\in\R^d,\
\xi\in \R^d,$$where $\lambda$ and $\Lambda$ are positive constants.
The energy associated to the matrix $a$ will be denoted by
\begin{equation}
\label{energy}
 \mathcal{E} \left( u,v\right):=\sum_{i,j=1}^d
\int_{\R^d}a^{i,j}(x)\partial_i u(x)\partial_j v(x)\, dx,\ \ \forall u,\, v \in  H^1(\R^d).
\end{equation}
\\As in the standard case we consider the space of test functions denoted by \\$\cD=\cC_c^\infty(\bbR^+)\times\cC_c^2(\R^d)$, where $\cC_c^\infty(\bbR^+)$ is the space of all real valued 
infinite differentiable functions with compact support in $\bbR^+$ and $\cC_c^2(\R^d)$ the set of $C^2-$functions with compact support in $\R^d$.\\
The set of solutions, denoted by $\cH_T^G$, is the set of processes in  \\$M^2_G ([0,T],(\cF^B_{t,T} )_{t\in [0,T]};H^1(\bbR^d))$ with $L^2 (\R^d )$-continuous trajectories equipped with the norm:
$$\forall u\in \cH_T^G, \|u\|_{\cH_T^G}^2=\bbE \bigg[ \sup_{t\in [0,T]}\| u_t\|^2\bigg] + \bbE \bigg[\int_0^T \| \nabla u_t\|^2 dt\bigg].$$
Clearly, $(\cH_T^G, \|\cdot\|_{\cH_T^G})$ is a Banach space.\\
 We fix the terminal condition $\Psi$ which is an element in $L^2 (\bbR^d)$ and make the following assumptions on coefficients $f$ and $g$.\\
 
{\bf Hypotheses (L):}
{\it $$f: [0,T]\times \Omega^B \times \R^d \times \R \times \R^d \longrightarrow \R$$ and
$$g:=\left( g^1,\cdots ,g^l \right): [0,T]\times \Omega^B \times \R^d \times \R \times \R^d \longrightarrow \R^l $$ are random functions, predictable with respect to the backward filtration  $(\cF^B_{t,T} )_{t\in [0,T]}$, satisfying the following Lipschitz conditions:  there exist positive constants $\bal$ and $\bC$ such that
\begin{enumerate}[leftmargin=*]
\item If $U$ and $V$ are processes belonging respectively to $M^2_G([0,T], (\cF_{t,T}^B )_{t\in [0,T]};L^2(\mathbb{R}^d))$ and $\left(M^2_G([0,T], (\cF_{t,T}^B)_{t\in [0,T]});L^2(\mathbb{R}^d)\right)^{ d}$ then $$(t,\omega)\in [0,T]\times \Omega\longrightarrow f(t,\omega,U_t(\omega),V_t(\omega))$$ and $$(t,\omega)\in [0,T]\times \Omega\longrightarrow g(t,\omega,U_t (\omega), V_t (\omega))$$  belong to $M^2_G ([0,T], (\cF_{t,T}^B)_{t\in [0,T]};L^2(\mathbb{R}^d))$ and $\left(M^2_G  ([0,T], (\cF_{t,T}^B)_{t\in [0,T]};L^2(\mathbb{R}^d))\right)^l$, respectively.
\item $|f(t,\omega,x,y,z)-f(t,\omega,x,y',z')|^2\leq \bC(|y-y'|^2+|z-z'|^2);$
\item $\sum_{j=1}^{l}|g^j(t,\omega,x,y,z)-g^j(t,\omega,x,y',z')|^2\leq\bC|y-y'|^2+\bal|z-z'|^2;$
\item the contraction property: $\bal\bar{\sigma}^2<2\lambda.$
\end{enumerate}}
\begin{remark} In the linear case, { i.e. the case where $\cP$ is a singleton}, the space $\hM^2_G([0,T], (\hF_t)_{t\in [0,T]};H)$ contains almost all the square integrable adapted processes. In the $G$-framework, this is not the case since for example the Fatou's Lemma is no more true. This explains assumption 1. in {\bf (L)}.  
\end{remark}
Finally, we introduce two kinds of solutions to GSPDEs \eqref{GSPDE}.
\begin{definition}
(Mild solution) We say that $u\in\cH_T^G$ is a
mild solution of the equation (\ref{GSPDE}) if the following
equality is verified quasi-surely: for each $t\in[0,T]$,
\begin{equation}\label{mildsolutionG}u_t=P_{T-t}\Psi+\int_t^T P_{s-t}f_s (\cdot, u_s ,\nabla u_s )ds+\int_t^TP_{s-t}g_s (\cdot, u_s ,\nabla u_s )\cdot \overleftarrow{dB}_s\,,\end{equation}
where $(P_t)_{t\geq0}$ is the $C_0-$semigroup associated to $L$: $\forall t\geq 0, \ P_t =e^{tL}$.  
\end{definition}
\begin{definition}
(Weak solution) We say that $u\in\cH_T^G$ is a
weak solution of  equation (\ref{GSPDE}) if the following
relation holds quasi-surely for each $\varphi\in\cD$,
\begin{equation}\label{weaksolutionG}\begin{split}&(u_t,\varphi_t)-(\Psi,\varphi_T)+\int_t^T(u_s,\partial_s\varphi_s)ds+\int_t^T\cE(u_s,\varphi_s)ds\\&=\int_t^T(f_s (\cdot, u_s ,\nabla u_s ),\varphi_s)ds+\int_t^T(g_s (\cdot, u_s ,\nabla u_s ),\varphi_s)\cdot \overleftarrow{dB}_s\,.\end{split}\end{equation}
\end{definition}

\subsection{The existence and uniqueness result for GSPDEs}
We start by showing that the quantities appearing in \eqref{mildsolutionG} are well-defined. 
\begin{lemma}{\label{lem1}}
Let $\Psi$ be in $L^2(\R^d)$. Then
\begin{enumerate}
\item $\alpha:\ t\in[0,T]\rightarrow P_{T-t}\Psi$ admits a  version in $L^2([0,T];H^1(\R^d))$  which is $L^2 (\R^d)$-continuous.
\item for all $\varphi\in\cD$ and for all $t\in[0,T]$, we have 
\begin{equation}\label{terminalvalueG}
(\Psi,\varphi_T)=(\alpha_t,\varphi_t)+\int_t^T(\alpha_s,\partial_s\varphi_s)ds+\int_t^T\cE(\alpha_s,\varphi_s)ds.
\end{equation}
\end{enumerate}
\end{lemma}
\begin{proof}
It is well-known that $\forall t\in[0,T)$, $P_{T-t}\Psi\in H^1(\bbR^d)$ and $$\forall t\in[0,T],\quad P_{T-t}\Psi=\int_0^\infty e^{-\lambda(T-t)}dE_\lambda\Psi,$$ where $\left( E_\lambda \right)_{\lambda >0}$ is the resolution of identity associated to $-L$. Therefore,
$$\forall t\in[0,T],\quad \|\alpha_t\|^2_{H^1(\bbR^d)}=\int_0^\infty(1+\lambda)e^{-2\lambda(T-t)}d(E_\lambda\Psi,\Psi),$$
this yields that 
\begin{equation*}\begin{split}
\int_0^T\|\alpha_t\|^2_{H^1(\bbR^d)}dt&=\int_0^\infty(1+\lambda)\int_0^Te^{-2\lambda(T-t)}dtd(E_\lambda\Psi,\Psi)\\&=\int_0^\infty(1+\lambda)\frac{1-e^{-2\lambda T}}{2\lambda}d(E_\lambda\Psi,\Psi)\\&\leq (T+\frac{1}{2})\|\Psi\|^2<\infty,
\end{split}\end{equation*}
which proves assertion 1.\\
To prove \eqref{terminalvalueG}, we first assume $\Psi\in\cD(L)$, then it is clear that 
the map $t\rightarrow P_{T-t}\Psi$ is $L^2-$differentiable and $\partial_t(P_{T-t}\Psi)=-L(P_{T-t}\Psi)$, 
hence, for all $\varphi\in\cD$, it holds that
\begin{equation*}\begin{split}
\partial_t(P_{T-t}\Psi,\varphi_t)&=-(L(P_{T-t}\Psi),\varphi_t)+(P_{T-t}\Psi,\partial_t\varphi_t)\\&=\cE(P_{T-t}\Psi,\varphi_t)+(P_{T-t}\Psi,\partial_t\varphi_t).
\end{split}\end{equation*}
Integrating from $t$ to $T$ yields that
$$(\Psi,\varphi_T)-(P_{T-t}\Psi,\varphi_t)=\int_t^T\cE(P_{T-s}\Psi,\varphi_s)ds+\int_t^T(P_{T-s}\Psi,\partial_s\varphi_s)ds.$$
\\For the general case, $\Psi\in L^2(\bbR^d)$, there
exists a sequence $\Psi^n$ in $\mathcal{D}(L)$ which converges to
$\Psi$ in $L^2(\bbR^d)$. Thanks to the proof of assertion 1, we know that
$(P_{T-t}\Psi^n)_{t\in [0,T]}$ converges to $(P_{T-t}\Psi )_{t\in [0,T]}$ in $L^2([0,T];H^1(\bbR^d))$ which yields the result by density. The $L^2 (\R^d )$-continuity is well known, see for example \cite{ste}.
\end{proof}

\begin{lemma}{\label{lem2}}
Let $\bar{f}$ be in $M_G^2([0,T],(\cF^B_{t,T} )_{t\in [0,T]};L^2(\R^d))$. Then
\begin{enumerate}
\item the process $\beta:\ t\in[0,T]\rightarrow\int_t^TP_{s-t}\bar{f}_sds$
admits a version in $\cH_T^G$ and there exists a constant $C$ depending only on $T$ and the structure constants of the GSPDE such that
$$\|\beta\|_{\cH_T^G}\leq C\|\bar{f}\|_{M^2_G([0,T],(\cF^B_{t,T} )_{t\in [0,T]};L^2(\R^d))}\, ;$$
\item for all $\varphi\in\mathcal{D}$ and all $t\in[0,T]$, we have 
\begin{equation}\label{relationbeta}(\beta_t,\varphi_t)+\int_t^T(\beta_s,\partial_s\varphi_s)ds+\int_t^T\mathcal{E}(\beta_s,\varphi_s)ds=\int_t^T(\bar{f}_s,\varphi_s)ds\ \ q.s.\end{equation}
\end{enumerate}
\end{lemma}
\begin{proof}
Assume first that $\bar{f}\in C^1([0,T])\otimes
L^2_G(\Omega^B)\otimes\mathcal{D}(L)$ and is $\cF^B_{t,T}-$adapted. Clearly, $C^1([0,T])\otimes
L^2_G(\Omega^B)\otimes\mathcal{D}(L)$ is dense in $M^2_G([0,T],(\cF^B_{t,T} )_{t\in [0,T]};L^2(\bbR^d))$. Fix $\omega\in\Omega^B$, for all $t\in[0,T],\
\beta_t(\omega)\in\mathcal{D}(L)$ and
$t\rightarrow\beta_t(\omega)$ is $L^2(\bbR^d)-$differentiable
and satisfies $$\forall\ t\in[0,T],\quad
\frac{d\beta_t}{dt}(\omega)=-\bar{f}_t(\omega)-L\beta_t(\omega).$$
Integrating by part we get, for all $\varphi\in\cD$ and all $t\in[0,T]$, 
$$(\beta_T,\varphi_T)-(\beta_t,\varphi_t)=\int_t^T(\beta_s,\partial_s\varphi_s)ds-\int_t^T( \bar{f}_s,\varphi_s)ds+\int_t^T\mathcal{E}(\beta_s,\varphi_s)ds.$$
Note that $\beta_T=0$, we obtain the desired relation.\\
Moreover, still integrating by part, we have, for all $t\in[0,T]$, 
$$\|\beta_t\|^2=-2\int_t^T(\partial_s\beta_s,\beta_s)ds=2\int_t^T(\bar{f}_s,\beta_s)ds-2\int_t^T\mathcal{E}(\beta_s)ds,$$
which yields
$$\|\beta_t\|^2+2\int_t^T\mathcal{E}(\beta_s)ds=2\int_t^T(\bar{f}_s,\beta_s)ds\leq\int_t^T(\| \bar{f}_s\|^2+\|\beta_s\|^2)ds.$$ 
Taking the supremum, we get quasi-surely for any $t\in [0,T]$
$$\sup_{u\in[t,T]}\|\beta_u\|^2\leq\int_0^T\| \bar{f}_s\|^2ds+\int_t^T\sup_{u\in[s,T]}\|\beta_u\|^2ds.$$
Thanks to the Grownall's lemma, we have
$$\sup_{t\in[0,T]}\|\beta_t\|^2\leq e^T\int_0^T\|\bar{f}_t\|^2dt,\quad q.s.$$
and
$$2\int_0^T\mathcal{E}(\beta_t)dt\leq\int_0^T\|\bar{f}_t\|^2+\|\beta_t\|^2dt\leq(1+Te^T)\int_0^T\| \bar{f}_t\|^2dt,\quad q.s.$$
Hence, we deduce that $$\| \beta\|_{\cH_T^G}^2=\bbE\sup_{t\in[0,T]}\|\beta_t\|^2+\bbE\int_0^T\cE(\beta_t)dt\leq(e^T+\frac{1}{2}+\frac{T}{2}e^T)\,\bbE\int_0^T\|\bar{f}_t\|^2dt.$$
Then the general case is obtained by a density argument.
\end{proof}


\begin{lemma}{\label{lem3}}
Let $\bar{g}:=( \bar{g}^1,...,\bar{g}^l)$ be in $\left(M_G^2([0,T],(\cF^B_{t,T} )_{t\in [0,T]};L^2(\R^d)\right)^l$. For all $0\leq t\leq s$ and $\varphi$ in $\cD$, we denote by $P_{s-t} \bar{g}_s:=(P_{s-t}\bar{g}^1_s,...,P_{s-t} \bar{g}^l_s)$ and $(\bar{g}_t,\varphi_t):=(( \bar{g}_t^1,\varphi_t),...,( \bar{g}_t^l,\varphi_t))$.  Then
\begin{enumerate}
  \item the process
  $\gamma:t\in[0,T]\rightarrow\int_t^TP_{s-t}\bar{g}_s\cdot\overleftarrow{dB}_s$ admits a version in $\cH_T^G$ and there exists a constant $C$ depending only on $T$ and the structure constants of the GSPDE
   such that \begin{equation}\label{firtsineq}\|\gamma\|_{\cH_T^G}\leq C\|  \bar{g}\|_{(M^2_G([0,T],(\cF^B_{t,T} )_{t\in [0,T]};L^2(\R^d)))^l};\end{equation} 
  \item for all $\varphi\in\cD$ and all $t\in [0,T]$, it holds 
  \begin{equation}\label{relationgamma}(\gamma_t, \varphi_t )+\int_t^T(\gamma_s,\partial_s\varphi_s)ds+\int_t^T\cE(\gamma_s,\varphi_s)ds=\int_t^T( \bar{g}_s,\varphi_s)\cdot\overleftarrow{dB}_s,\ \
  q.s.\end{equation}
\end{enumerate}
\end{lemma}

\begin{proof}
We denote by $\cS$ the set of processes $\bar{g}$ such that
$$\forall\,(t,x,\omega)\in[0,T]\times\R^d\times\Omega^B,\ \bar{g}(t,x,\omega)=\sum_{i=0}^{n-1} \bar{g}_{i+1}(x,\omega)I_{(t_i,t_{i+1}]}(t),$$
where $n\in\bbN^*$, $0=t_0\leq t_1\leq...\leq t_n=T$ and for all $i\in\{0,1,...,n-1\}$,
$$\forall (x,\omega)\in \R^d\times\Omega^B,\ \ \bar{g}_{i+1}(x,\omega)=\sum_{j=1}^{n_i}\kappa_{i+1}^j(\omega)\bar{g}_{i+1}^j(x)\,,$$where
$n_i\in\bbN^*$ and for all $j\in\{1,...,n_i\}$,
$\bar{g}_{i+1}^j\in\cD(L)$ and $\kappa_{i+1}^j\in L^2_G(\Omega^B)$ is $\cF^B_{t_{i+1},T}$-measurable. As $\cD(L)$
is dense in $L^2(\R^d)$, we can easily prove that $\cS$ is a dense
subspace in $M_G^{2,0}([0,T],(\cF^B_{t,T} )_{t\in [0,T]};L^2(\R^d))$ hence in
$M_G^{2}([0,T],(\cF^B_{t,T} )_{t\in [0,T]};L^2(\R^d))$.
\\Assume first that $\bar{g}\in\cS^l$. It is clear that the process
$$\forall t\in[0,T],\ \gamma_t=\int_t^TP_{s-t} \bar{g}_s\cdot\overleftarrow{dB}_s$$ admits a version in $\cH_T^G$ and  $(L\gamma_t )_{t\in [0,T]}$ belongs to  $M_G^{2,0}([0,T],(\cF^B_{t,T} )_{t\in [0,T]};L^2(\R^d))$. Under each $P\in\cP$, a direct calculation yields for any $t\in [0,T]$:
\begin{equation*}\begin{split}
L\int_t^T\gamma_s ds=&\int_t^T L\gamma_s ds=\int_t^T \int_s^T LP_{u-s}\bar{g}_u \cdot\overleftarrow{dB}_u ds\\
=&\sum_{j=1}^l\int_t^T\int_t^u  LP_{u-s}\bar{g}_u ^jds\,\overleftarrow{dB}_u^j =-\sum_{j=1}^l\int_t^T\int_t^u  \frac{\partial}{\partial s}P_{u-s}\bar{g}_u ^jds\,\overleftarrow{dB}_u^j \\
=&\int_t^T(P_{u-t}\bar{g}_u -\bar{g}_u )\cdot\overleftarrow{dB}_u,
\end{split}\end{equation*}
so that we have
$$\gamma_t=\int_t^T\bar{g}_s\cdot\overleftarrow{dB}_s+\int_t^TL\gamma_sds\,.$$
Integrating by parts, we have, for all $\varphi\in\cD$ and all $t\in [0,T]$:
\begin{equation*}\begin{split}
 (\gamma_t ,\varphi_t )&=-\int_t^T(\gamma_s,\partial_s\varphi_s ) ds+\int_t^T( \bar{g}_s,\varphi_s)\cdot\overleftarrow{dB}_s+\int_t^T(L\gamma_s,\varphi_s)ds\\&=-\int_t^T(\gamma_s,\partial_s\varphi_s)ds+\int_t^T(\bar{g}_s,\varphi_s)\cdot\overleftarrow{dB}_s-\int_t^T\cE(\gamma_s,\varphi_s)ds\,,
\end{split}\end{equation*}
 which proves  relation \eqref{relationgamma} in that case.\\
 It\^o's formula yields for all $t\in[0,T]$,
\begin{eqnarray}\label{ito}
\left\|\gamma_t\right\|^2+2\int_t^T\cE(\gamma_s)ds=2\int_t^T(\gamma_s, \bar{g}_s)\cdot\overleftarrow{dB}_s+\sum_{i,j=1}^l\int_t^T\bar{g}^i_s\bar{g}^j_s\, d\langle \tilde{B}^i,\tilde{B}^j\rangle_s\,,
\end{eqnarray}
$P$-almost surely for any $P\in \cP$. Taking the supremum over $\cP$, we obtain
\begin{eqnarray*}
\bbE\bigg[\sup_{t\in[0,T]}\left\|\gamma_t\right\|^2\bigg]\leq
2\bbE\bigg[\sup_{t\in[0,T]}\Big|\int_t^T(\gamma_s, \bar{g}_s)\overleftarrow{dB}_s\Big|\bigg]+\bar{\sigma}^2\bbE\bigg[\int_0^T\left\|\bar{g}_s\right\|^2ds\bigg]\,.
\end{eqnarray*}
Using Burkholder-David-Gundy's inequality (see \cite{Gao}) there exists a constant $C_1$ such that:
\begin{equation*}\begin{split}
\bbE\bigg[\sup_{t\in[0,T]}\left\|\gamma_t\right\|^2\bigg]\leq&\, 
2C_1 \bar{\sigma}^2\bbE\bigg[\Big(\int_0^T(\gamma_s,\bar{g}_s)^2ds\Big)^{1/2}\bigg]
+\bar{\sigma}^2\bbE\bigg[\int_0^T\left\|\bar{g}_s\right\|^2ds\bigg]\\\leq&\,
2C_1 \bar{\sigma}^2\bbE\bigg[\sup_{t\in[0,T]}\|\gamma_t\|\times\Big(\int_0^T\| \bar{g}_s\|^2ds\Big)^{1/2}\bigg]+\bar{\sigma}^2\bbE\bigg[\int_0^T\left\|\bar{g}_s\right\|^2ds\bigg],
\end{split}\end{equation*}
so for any $\epsilon>0$ we have
\begin{eqnarray*}
\bbE\bigg[\sup_{t\in[0,T]}\left\|\gamma_t\right\|^2\bigg]\leq
C_1 \bar{\sigma}^2\epsilon\,\bbE\bigg[\sup_{t\in[0,T]}\|\gamma_t\|^2\bigg]+\big(\frac{C_1 \bar{\sigma}^2}{\epsilon}+\bar{\sigma}^2\big)\bbE\bigg[\int_0^T\left\|\bar{g}_s\right\|^2ds\bigg]\,.
\end{eqnarray*}
Then we can take $\epsilon$ such that ${C_1 \bar{\sigma}^2}{\epsilon}\leq 1/2$, and get 
$$\bbE\bigg[\sup_{t\in[0,T]}\|\gamma_t\|^2\bigg]\leq  2\big(\frac{C_1 \bar{\sigma}^2}{\epsilon}+\bar{\sigma}^2\big)\bbE\bigg[\int_0^T\left\|\bar{g}_s\right\|^2ds\bigg]\,.$$
These estimates ensure that $E_P [\int_t^T(\gamma_s, \bar{g}_s)\cdot\overleftarrow{dB}_s ]=0$ for any $P\in\cP$, hence by relation (\ref{ito}),
$$\bbE\int_0^T\cE(\gamma_s)ds\leq \displaystyle\frac{\bar{\sigma}^2}{2}\bbE\bigg[\int_0^T\left\| \bar{g}_s\right\|^2ds\bigg],$$
which proves inequality \eqref{firtsineq} in this case.\\
We conclude by a density argument in the general case. 
\end{proof}

\begin{proposition}
Notions of {\rm mild solutions} and {\rm weak solutions} coincide.
\end{proposition}
\begin{proof}
Let $u\in\mathcal{H}_T^G$ be a mild solution of equation \eqref{GSPDE}, we have for all $t\in [0,T]$:
\begin{equation}\label{mildequation}\begin{split}
u(t,x)=&\,P_{T-t}\Psi(x)+\int_t^TP_{s-t}f(s,\cdot,u(s,\cdot),\nabla u(s,\cdot))(x)ds\\&+\int_t^TP_{s-t}g(s,\cdot,u(s,\cdot),\nabla u(s,\cdot))(x)\cdot\overleftarrow{dB}_s.
\end{split}\end{equation}
Combining \eqref{terminalvalueG}, \eqref{relationbeta} and \eqref{relationgamma}, we can easily deduce that $u$  is a weak solution of \eqref{GSPDE}. 
Conversely, let $u\in\mathcal{H}_T^G$ be a weak solution of \eqref{GSPDE}. Then define the process $\bar{u}\in\mathcal{H}_T$ by:
\begin{equation*}\begin{split}
\forall t\in [0,T],\ \ \bar{u}(t,x)=&\,P_{T-t}\Psi(x)+\int_t^TP_{s-t}f(s,\cdot,u(s,\cdot),\nabla u(s,\cdot))(x)ds\\&+\int_t^TP_{s-t}g(s,\cdot,u(s,\cdot),\nabla u(s,\cdot))(x)\cdot\overleftarrow{dB}_s.
\end{split}\end{equation*}
As a consequence of Lemmas \ref{lem1}, \ref{lem2} and \ref{lem3}, we have for all $\varphi\in\mathcal{D}$:
\begin{equation*}
\begin{split}
\int_0^T(\bar{u}_s,\partial_s\varphi_s)ds+\int_0^T\cE(\bar{u}_s,\varphi_s)ds&=\\(\Psi,\varphi_T)+\int_0^T(f_s& (u_s ,\nabla u_s ),\varphi_s)ds+\int_0^T(g_s (u_s ,\nabla u_s ),\varphi_s)\cdot \overleftarrow{dB}_s\,.
\end{split}
\end{equation*}
On the other hand, as $u$ is a weak solution of \eqref{GSPDE}, for any $\varphi\in\mathcal{D}$,
\begin{equation*}
\begin{split}
\int_0^T(u_s,\partial_s\varphi_s)ds+\int_0^T\cE(u_s,\varphi_s)ds=&\\(\Psi,\varphi_T)+\int_0^T(f_s &(u_s ,\nabla u_s ),\varphi_s)ds+\int_0^T(g_s (u_s ,\nabla u_s ),\varphi_s)\cdot \overleftarrow{dB}_s\,.\end{split}
\end{equation*}
If we put $v_t(x)=u_t(x)-\bar{u}_t(x)$, it is clear that $v\in\mathcal{H}_T$ and that for any $\varphi\in\mathcal{D}$, $$\int_0^T(v_t,\partial_t\varphi_t)dt+\int_0^T\cE(v_t,\varphi_t)dt=0,\quad q.s.$$
This means that, under each $P\in\cP$, $v$ is the weak solution of the equation $\partial_tv_t+Lv_t=0$ with terminal condition $u_T=0$. Then, thanks to Lemma 4.10 in \cite{Denis}, we know that, under each $P\in\cP$, $v = 0$, $P-$a.s., hence, $v=0$, q.s.
\end{proof}

The main result of this part is the following theorem:
\begin{theorem}\label{existuniqch4}
Under hypotheses (\textbf{L}), equation \eqref{GSPDE} admits a unique weak hence mild  solution in
$\mathcal{H}_T^G$ that we denote  by $\cG(\Psi,f,g)$.
\end{theorem}
To prove this theorem, we need the following It\^o's formula: 
\begin{theorem}\label{itoformula}{\bf (It\^o's formula)} Assume that $\bar{f}\in M_G^2([0,T],(\cF^B_{t,T} )_{t\in [0,T]};L^2(\R^d))$, $\bar{g}=(\bar{g}^1,...,\bar{g}^l)\in \left(M_G^2([0,T],(\cF^B_{t,T} )_{t\in [0,T]};L^2(\R^d)\right)^l$ and $\Psi\in L^2(\R^d)$. Consider 
\begin{equation}\label{mildsolutionGlinear}\forall t\in [0,T],\ \ u_t(x)=P_{T-t}\Psi(x)+\int_t^TP_{s-t}\bar{f}_s(x)ds+\int_t^TP_{s-t}\bar{g}_s(x)\cdot\overleftarrow{dB}_s.\end{equation}
Let $\Phi:\bbR^+\times\bbR\rightarrow\bbR$ be a function of class $\cC^{1,2}$. Denote by $\Phi'$ and $\Phi''$ the first and second order derivatives of $\Phi$ with respect to the space variable and by 
$\frac{\partial\Phi}{\partial t}$ the derivative with respect to the time variable and assume that $\Phi''$ is bounded. Then the following relation holds quasi-surely for all $t\in[0,T]$,
\begin{equation}\label{itoformulaGSPDEch4}\begin{split}
\int_{\R^d}\Phi(t,u_t(x))dx+&\int_t^T\cE(\Phi'(s,u_s),u_s)ds= \\\int_{\R^d}\Phi(T,\Psi(x))dx &
-\int_t^T\int_{\R^d}\frac{\partial\Phi}{\partial s}(s,u_s(x))dxds+\int_t^T(\Phi'(s,u_s),\bar{f}_s)ds\\+\sum_{j=1}^{l}\int_t^T(\Phi'(s,u_s),\bar{g}_s^j)\overleftarrow{dB}^j_s&+\frac{1}{2}\sum_{i,j=1}^{l}\int_t^T\int_{\R^d}\Phi''(s,u_s(x))\bar{g}^i_s(x)\bar{g}^j_s(x)dxd\langle \tilde{B}^i, \tilde{B}^j\rangle_s\,.
\end{split}\end{equation}
\end{theorem}
\begin{proof}
We begin with the regular case, i.e. $\bar{f},\ \bar{g}^j\in C^1([0,T])\otimes L^2_G(\Omega_T^B)\otimes\cD(L)$ are $\cF^B_{t,T}-$adapted, and 
$\Psi\in\cD(L)$, then following the proofs of Lemmas  \ref{lem1} to \ref{lem3} we get that under each probability $P\in \cP$, $u$ is a $L^2 (\R^d )$-valued  backward semi-martingale admitting the following decomposition:
\begin{equation*}
u_t=\Psi+\int_t^TLu_sds+\int_t^T\bar{f}_sds+\int_t^T\bar{g}_s\cdot\overleftarrow{dB}_s\,.
\end{equation*}
Thanks to It\^o's formula for semi-martingale, we have $P$-almost surely for all $P\in\cP$ hence quasi-surely for all $t\in[0,T]$: 
\begin{equation*}\begin{split}
\int_{\R^d}\Phi(t,u_t(x))dx=&\,\int_{\R^d}\Phi(T,\Psi(x))dx+\int_t^T(\Phi'(s,u_s),Lu_s)ds
\\-\int_t^T\int_{\R^d}&\frac{\partial\Phi}{\partial s}(s,u_s(x))dxds+\int_t^T(\Phi'(s,u_s),\bar{f}_s)ds\\+\int_t^T(\Phi'(s,u_s),\bar{g}_s)\cdot\overleftarrow{dB}_s&+\frac{1}{2}\sum_{i,j=1}^{l}\int_t^T\int_{\R^d}\Phi''(s,u_s(x))\bar{g}^i_s(x)\bar{g}^j_s(x)dxd\langle \tilde{B}^i, \tilde{B}^j\rangle_s\,.
\end{split}\end{equation*}
Then, as $(\Phi'(s,u_s),Lu_s)=-\cE(\Phi'(s,u_s),u_s)$, we get the desired equality.\\
The general case is obtained by approximation. We take $\bar{f}^n,\ \bar{g}^n\in C^1([0,T])\otimes L^2_G(\Omega_T^B)\otimes\cD(L)$
and $\Psi^n\in\cD(L)$ such that $f^n\rightarrow \bar{f}$, $\bar{g}^n\rightarrow \bar{g}$ strongly in \\$M^2_G([0,T],(\cF_{t,T}^B)_{t\in[0,T]};H^1(\R^d))$ and $\Psi^n\rightarrow\Psi$ strongly in $L^2(\R^d)$. 
Therefore, $$u^n_t(x)=P_{T-t}\Psi^n(x)+\int_t^TP_{s-t}\bar{f}_s^n(x)ds+\int_t^TP_{s-t}\bar{g}_s^n(x)\cdot\overleftarrow{dB}_s$$ 
converges strongly to $u$ defined in \eqref{mildsolutionGlinear} in $M^2_G([0,T],(\cF_{t,T}^B)_{t\in[0,T]};H^1(\R^d))$. Thanks to the first step, 
we have the It\^o formula for $u^n$, that is, quasi-surely for any $t\in[0,T]$,
\begin{equation*}\begin{split}
\int_{\R^d}\Phi(t,u_t^n(x))dx&+\int_t^T\mathcal{E}(\Phi'(s,u_s^n),u_s^n)ds=\,\int_{\R^d}\Phi(T,\Psi^n(x))dx
\\-\int_t^T\int_{\R^d}&\frac{\partial\Phi}{\partial s}(s,u^n_s(x))dxds+\int_t^T(\Phi'(s,u^n_s),\bar{f}^n_s)ds\\+\int_t^T(\Phi'(s,u^n_s),\bar{g}^n_s)\cdot\overleftarrow{dB}_s&+\frac{1}{2}\sum_{i,j=1}^{l}\int_t^T\int_{\R^d}\Phi''(s,u^n_s(x))\bar{g}^{n,i}_s(x)\bar{g}^{n,j}_s(x)dxd\langle \tilde{B}^i, \tilde{B}^j\rangle_s\,.
\end{split}\end{equation*}
Then, under each $P\in\cP$, thanks to the dominated convergence theorem under $P$, the following It\^o formula hold $P-$a.s., (see for example Lemma 7 in \cite{DMS05})
\begin{equation*}\begin{split}
\int_{\R^d}\Phi(t,u_t(x))dx&+\int_t^T\mathcal{E}(\Phi'(s,u_s),u_s)ds=\,\int_{\R^d}\Phi(T,\Psi(x))dx
\\-\int_t^T\int_{\R^d}&\frac{\partial\Phi}{\partial s}(s,u_s(x))dxds+\int_t^T(\Phi'(s,u_s),\bar{f}_s)ds\\+\int_t^T(\Phi'(s,u_s),\bar{g}_s)\cdot\overleftarrow{dB}_s&+\frac{1}{2}\sum_{i,j=1}^{l}\int_t^T\int_{\R^d}\Phi''(s,u_s(x))\bar{g}^i_s(x)\bar{g}^j_s(x)dxd\langle \tilde{B}^i, \tilde{B}^j\rangle_s\,.
\end{split}\end{equation*}
Finally, as each member in the equality admits a quasi-continuous version, the formula holds quasi-surely. 
\end{proof}

Now we come to the proof of {\bf Theorem \ref{existuniqch4}}:
\begin{proof}
Let $\gamma$ and $\delta$ be two positive constants. On $M^2_G([0,T],(\cF^B_{t,T} )_{t\in [0,T]};H^1(\R^d))$, we introduce the following norm 
$$\left\| u\right\|_{\gamma,\delta}=\bbE\bigg(\int_0^Te^{\gamma s}\left(\delta\left\| u_s\right\|^2+\left\|\nabla u_s\right\|^2\right)ds\bigg),$$
which clearly defines an equivalent norm on $M^2_G([0,T],(\cF^B_{t,T} )_{t\in [0,T]};H^1(\R^d))$. \\
We define the application $$\Lambda: M^2_G([0,T],(\cF^B_{t,T} )_{t\in [0,T]};H^1(\R^d))\rightarrow M^2_G([0,T],(\cF^B_{t,T} )_{t\in [0,T]};H^1(\R^d))$$ as follows:
\begin{equation*}
\forall t\in [0,T],\ (\Lambda u)_t=P_{T-t}\Psi+\int_t^TP_{s-t}f_s(u_s,\nabla u_s)ds+\int_t^TP_{s-t}g_s(u_s,\nabla u_s)\cdot\overleftarrow{dB}_s\,.
\end{equation*}
Denoting $\bar{u}_t=\Lambda u_t-\Lambda v_t$ with $u$ and $v$ in
$M^2_G([0,T],(\cF^B_{t,T} )_{t\in [0,T]};H^1(\R^d))$, 
$\bar{f}=f(u,\nabla u)-f(v,\nabla v)$ and $\bar{g}=g(u,\nabla u)-g(v,\nabla v)$. By using It\^o's formula (Theorem \ref{itoformula}),  
we have quasi-surely: 
\begin{equation}\label{ItoEx}\begin{split}
\|\bar{u}_0\|^2+2\int_0^Te^{\gamma s}\cE(\bar{u}_s)ds=&\,{- \gamma}\int_0^Te^{\gamma s}\|\bar{u}_s\|^2ds
+2\int_0^Te^{\gamma s}(\bar{u}_s,\bar{f}_s)ds\\
+2\sum_{j=1}^l\int_0^Te^{\gamma s}&(\bar{u}_s,\bar{g}^j_s)\overleftarrow{dB}^j_s+\sum_{i,j=1}^l\int_0^Te^{\gamma s}(\bar{g}^i_s,\bar{g}^j_s)d\langle \tilde{B}^i,\tilde{B}^j\rangle_s\,.
\end{split}\end{equation}
Let $P\in\cP$, then since under $P$,  $(d\langle \tilde{B}^i,\tilde{B}^j\rangle_t)_{i,j}\leq\bar{\sigma}^2I_ddt$, $P-$a.s., we get for any $j$:
\begin{equation*}\begin{split}
    2E_P\bigg[\bigg( \int_0^T(\bar{u}_s,\bar{g}^j_s)^2 d\langle \tilde{B}^j,\tilde{B}^j\rangle_s\bigg)^{1/2}\bigg]&\leq 2\bar{\sigma}E_P\bigg[\bigg( \int_0^T (\bar{u}_s,\bar{g}^j_s)^2 ds\bigg)^{1/2}\bigg]\\
    &\leq 2\bar{\sigma}E_P\bigg[\bigg( \sup_{s\in [0,T]}\|\bar{u}_s\|\int_0^T \| \bar{g}^j_s\|^2 ds\bigg)^{1/2}\bigg]\\
    &\leq\bar{\sigma}E_P\bigg[ \sup_{s\in [0,T]}\|\bar{u}_s\|^2 + \int_0^T \| \bar{g}^j_s\|^2 ds\bigg]<+\infty. 
\end{split}\end{equation*}
Hence, $E_P\big[\sum_{j=1}^l\int_0^Te^{\gamma s}(\bar{u}_s,\bar{g}^j_s)\overleftarrow{dB}^j_s\big]=0$.\\
 Thanks to Lipschitz conditions satisfied by coefficients $f$ and $g$, Assumption {\bf (L)} and the elliptic condition on operator $L$ we have by taking expectation under $P$ in \eqref{ItoEx}: 
\begin{equation*}\begin{split}
& E_P \bigg[ \gamma\int_0^Te^{\gamma s}\|\bar{u}_s\|^2ds+2\lambda \int_0^Te^{\gamma s}\| \nabla \bar{u}_s\|^2ds\bigg]\\ \leq&\,
2E_P \bigg[ \int_0^Te^{\gamma s}\bar{C}^{1/2}\|\bar{u}_s\|\left(\|u_s -v_s\|^2+\|\nabla({u}_s-v_s)\|^2 \right)^{1/2}ds\bigg]\\&
+\bar{\sigma}^2 E_P\bigg[ \int_0^Te^{\gamma s}\left( \bar{C}\|u_s -v_s\|^2 +\bar{\alpha} \|\nabla (u_s-v_s)\|^2\right) ds\bigg]\,.
\end{split}\end{equation*}
Let $\epsilon >0$, using one more time the trivial inequality $2ab\leq \displaystyle\frac{a^2}{\epsilon}+\epsilon b^2$, we get:
\begin{eqnarray*}&&\hspace{-0.3cm}E_P\Big[(\gamma-\frac{1}{\epsilon})\int_0^Te^{\gamma
s}\|\bar{u}_s\|^2ds+2\lambda\int_0^Te^{\gamma
s}\|\nabla\bar{u}_s\|^2ds\Big]\\&&\hspace{-0.7cm}\leq
\bar{C}\left(\bar{\sigma}^2+\epsilon\right)E_P\int_0^Te^{\gamma s}\|
\bar{u}_s\|^2ds+\big(\bar{C}\epsilon + \bar{\sigma}^2\bar{\alpha}\big)E_P\int_0^Te^{\gamma
s}\|\nabla\bar{u}_s\|^2ds.\end{eqnarray*}

Thanks to Assumption {\bf (L)} 3., we can  choose $\epsilon$ small enough and then $\gamma$ such that
$$0<\displaystyle\frac{\bar{C}\epsilon + \bar{\sigma}^2\bar{\alpha}}{2\lambda}<1\quad \mbox{and}\quad \frac{\gamma-1/\epsilon}{2\lambda}=\displaystyle\frac{\bar{C}(\bar{\sigma}^2+\epsilon)}{\bar{C}\epsilon + \bar{\sigma}^2\bar{\alpha}}
$$
If we set $ \kappa= \displaystyle\frac{\bar{C}\epsilon + \bar{\sigma}^2\bar{\alpha}}{2\lambda}$ and $\delta=\displaystyle\frac{\gamma-1/\epsilon}{2\lambda}$, we have  for any $P\in\cP$: 
$$E_P\bigg[\int_0^Te^{\gamma s}(\delta\|\bar{u}_s\|^2+\|\nabla\bar{u}_s\|^2)ds\bigg]\leq\kappa E_P\bigg[\int_0^Te^{\gamma s}(\delta\|u_s-v_s\|^2+\|\nabla(u_s-v_s)\|^2)ds\bigg].$$
Then taking supreme over $\cP$ provides that
$$\|\bar{u}\|_{\gamma,\delta}\leq\kappa \|
u-v\|_{\gamma,\delta}.$$
We conclude thanks to the fixed point theorem.
\end{proof}

\subsection{Comparison theorem}
In this subsection we present a comparison theorem for the solution of GSPDE \eqref{GSPDE}. We still assume {\bf (L)} and  consider $u=\cG(\Psi,f,g)$.
\begin{theorem}
Let $f'$ be another coefficient which satisfies the same hypotheses as $f$ and $\Psi'\in L^2(\R^d)$. Let $u'$ be the solution of 
\begin{eqnarray*}du'_t(x)+Lu'_t(x)dt+f'(t,x,u'_t(x),\nabla
u'_t(x))dt+\sum_{j=1}^{l}g_j(t,x,u'_t(x),\nabla
u'_t(x))\overleftarrow{dB}^j_t=0,\end{eqnarray*} 
with initial condition $u'_T=\Psi'$. \\ Assume that $\Psi\leq\Psi',\ a.e.$ and for quasi all $\omega\in\Omega$,  
$$f(t,x,u_t(x),\nabla u_t(x))\leq f'(t,x,u_t(x),\nabla u_t(x)),\ dt\otimes dx-a.e.,$$
then $$\forall t\in[0,T],\ \ u_t\leq u'_t\,,\ q.e.$$
\end{theorem}
\begin{proof}
The comparison theorem for SPDE's (see, for example, Theorem 5.2 in \cite{Denis}) ensures that for all $P\in\cP$, $$\forall t\in[0,T],\ \ u_t\leq u'_t,\ \ P-a.s.$$
We conclude thanks to the quasi-continuity of $u$ and $u'$. 
\end{proof}

\section{The well-posedness of GBDSDEs}
\subsection{The associated forward process}{\label{ForwardP}}
For this subsection, we mainly refer to \cite{FOT} and \cite{BPS}. Evoke that $L=\sum_{i,j=1}^d\partial_i(a^{i,j}\partial_j)$ is a symmetric strictly elliptic second-order differential operator associated  to the semi-group $(P_t)_{t\geq0}$ and to the bilinear form 
$$\cE(u,v):=\sum_{i,j}\int_{\R^d}a^{i,j}(x)\frac{\partial u}{\partial x^i}(x)\frac{\partial v}{\partial x^j}(x)dx,\quad u,v\in H^1  (\bbR^d).$$
 We also define
$$\cE_1(u,v):=\cE(u,v)+(u,v).$$
 Clearly $(H^1 (\R^d ),\cE_1 )$ is a local Dirichlet form so it is well-known that there exists a symmetric Markov process $(\Omega^W,\mathcal{F}^W,\{\mathcal{F}^W_t\}_{t\geq0},\{X_t\}_{t\geq0},P^x,x\in\mathbb{R}^d)$ associated to the Dirichlet form $(H^1  (\bbR^d),\cE)$. This process is continuous because $\cE$ is local and we know that all $\{\mathcal{F}^W_t\}_{t\geq0}$  martingales are continuous if  $\{\mathcal{F}^W_t\}_{t\geq0}$ is the filtration generated by  the process
$ (X_t)_{t\geq0}$ in a canonical way. 
As a consequence, we can consider  the canonical version of the process $\{X_t \}_{t\geq 0}$, this means that we take $\Omega^W:=C([0,T];\bbR^d)$, the set of continuous functions from $[0,T]$ into $\R^d$ and $\{X_t \}_{t\geq 0}$ is  the coordinates process.\\ 
We recall that $\{X_t \}_{t\geq 0}$ is the Markovian process with generator $L$  such that for any $x\in\R^d$, $P^x (X_0 =x)=1.$\\ 
This Dirichlet form on $\R^d$ and so the process  $\{X_t\}_{t\geq0}$ defines a capacity on $\R^d$ that we denote by $c_\cE$, see for example \cite[Chapter I.8]{bouleau:hirsch}.
\begin{definition}
A borelian set $A\subset\R^d$ is called $c_\cE$-polar if $c_\cE(A)=0$. A property is said to hold ``$c_\cE$-quasi-surely" (q.s.) if it holds outside a  $c_\cE$-polar set.
\end{definition}
\begin{definition}
A mapping $\xi$ defined on $\R^d$ with values in a topological space is said to be $c_\cE$-quasi-continuous ($c_\cE$-q.c.) if 
$\forall\epsilon>0$, there exists an open set $O$ with $c_\cE(O)<\epsilon$ such that $\xi|_{O^c}$ is continuous. 
\end{definition}
Following \cite{FOT}, it is known that every element $u\in H^1 (\R^d )$ admits a $c_\cE$-quasi-\\continuous version denoted by $\tilde{u}$ and that for $c_\cE$-quasi all $x$ in $\R^d$, the process $\{\tilde{u}(X_t)\}_{t\geq 0}$ is continuous $P^x$-a.s. From now on, even if we do not mention it, we shall only consider $c_{\cE}$-quasi-continuous version of elements in $H^1 (\R^d )$.\\
We introduce the invariant measure of the process $\{X_t\}_{t\geq 0}$: $P^m:=\int_{\R^d }P^x \, dx$.\\
If $u\in\mathcal{D}(L)$, then it admits the following decomposition
$$u(X_t)-u(X_0)=M_t^u+N_t^u,$$
where $M^u$ is a martingale under $P^m$ but also under $P^x$ for $c_\cE$-quasi every  $x$ and $N_t^u=\int_0^tLu_sds$. The above decomposition also holds for $u\in H^1 (\R^d )$ with $N^u$  a process of null energy.  Moreover the decomposition holds true for $u$ in $ H^1_{loc} (\R^d )$, the local Sobolev space, in that case $M^u$ is only a local martingale. The coordinate functions $u^i(x)=x^i$, $i=1,\cdots,d$ are in $H^1_{loc} (\R^d )$ and we denote $M^i:=M^{u^i}$.\\
For any $x\in\R^d$, we define $\sigma (x)=\left( a(x)\right)^{1/2}$, the square root of  $a(x)$. It is known that it is a symmetric, definite positive matrix. Note that if  $\sigma $ is $C^2$ with bounded derivatives, $X$ may be viewed as the solution of the stochastic differential equation $$dX_t^i=\sqrt{2}\sum_{j=1}^d\sigma^{i,j}(X_t)dW_t^j+\sum_{j=1}^d \partial_j a^{i,j}(X_t) dt,$$  
where $W:=(W^1,\cdots,W^d)$ is a $d-$dimensional Brownian motion. Hence the martingale parts of the coordinates are represented as $$M_t^i=\sqrt{2}\int_0^t\sum_{j=1}^n\sigma^{i,j}(X_s)dW_s^j.$$
In the general case we consider here, such a representation may not make sense but following \cite[Example 5.2.1]{FOT}, we know that under $P^x$ for $c_\cE$-quasi every point $x\in\bbR^d$, 
\begin{equation}\label{quadraticvaritionM}
\langle M^i,M^j\rangle_t=2\int_0^t\sum_{m=1}^d\sigma^{i,m}(X_s)\sigma^{j,m}(X_s)ds=2\int_0^t a^{i,j}(X_s)ds,\ \ P^x-\makebox{a.s}.
\end{equation}
If $\phi=(\phi^1,\cdots,\phi^d):[0,T]\times\Omega^W\rightarrow\bbR^d$ is $(\cF^W_t )_{t\in [0,T]}$-predictable and bounded, the martingale $$\int_0^t\phi_s\cdot dM_s:=\sum_{i=1}^d\int_0^t\phi^i_s\,dM_s^i$$ 
is well-defined $P^x$-almost surely for $c_\cE$-quasi every  $x\in\bbR^d$ and its  bracket satisfies:
$$\left\langle\int_0^\cdot \phi^i_s \cdot dM_s ,\int_0^\cdot \phi^i_s \cdot dM_s\right\rangle_t=2\int_0^t (\phi_s \cdot a(X_s)\cdot \phi_s^* ) \, ds{{\leq 2\Lambda\int_0^t |\phi_s|^2 ds}}.$$
By an approximation procedure, if $\phi$ is $(\cF^W_t )_{t\in [0,T]}$-predictable and in $L^2 ([0,T]\times \Omega^W ,dt\otimes P^m)$ the process
$$t\in [0,T]\rightarrow \int_0^t\phi_s \cdot dM_s$$
is well-defined, belongs to $L^2 ([0,T]\times \Omega^W ,dt\otimes P^m)$ and satisfies
$$E_{P^m} \bigg[\bigg|\int_0^T \phi_s \cdot dM_s\bigg|^2\bigg]=2E_{P^m}\bigg[\int_0^T (\phi_s \cdot a(X_s)\cdot \phi_s^* ) \, ds\bigg] ,$$
as a consequence of the ellipticity condition on the matrix $a$:
\begin{eqnarray}{\label{IneqBracket}}
2\lambda E_{P^m}\bigg[\int_0^T |\phi_s|^2  \, ds\bigg]\leq E_{P^m} \bigg[\bigg|\int_0^T \phi_s \cdot dM_s\bigg|^2\bigg]\leq2\Lambda E_{P^m}\bigg[\int_0^T |\phi_s|^2  \, ds\bigg].
\end{eqnarray}

\subsection{The product space on which we shall work}
 We consider the product space $\Omega:=\Omega^W \times \Omega^B$ and in a natural way, from now on, we consider that $B$, $\Bt =B_{T-\cdot}-B_{T}$, $X$ and $M$ are defined on $\Omega$ by 
$$\forall \omega:=(\omega^1,\omega^2)\in\Omega,\ \forall t\in [0,T],\ X_t(\omega)= X_t(\omega^1),\ M_t (\omega)=M_t (\omega^1)  \makebox{ and }B_t (\omega)=\omega^2_t.$$
 We introduce the following families of $\sigma$-fields on $\Omega $:
 $$\forall t\in [0,T],\ \hF^W_t:=\sigma (X_s;\ s\leq t),\ \hF^B_t:=\sigma (B_s;\  s\leq t);$$ and $$\hF^{B}_{t,T}:=\sigma (B_s-B_{t},\  t\leq s\leq T)=\sigma (\Bt_r;\  r\leq T-t),$$
 then we define $\hF_t := \hF^W_t \vee \hF^{B}_{t,T},\ \forall t\in [0,T].$\\
 Let us recall that $(\hF^W_t)_{t\in [0,T]}$, $(\hF^B_t)_{t\in [0,T]}$ and $(\hF^B_{T-t,T})_{t\in [0,T]}$ are filtrations whereas $(\hF_t)_{t\in [0,T]}$ is neither increasing nor decreasing.\\
 At some point, we shall also need the following ``enlarged'' filtrations:
 $$\forall t\in [0,T],\ \ \hG^W_t:=\hF^W_t \vee \hF^B_T\ \makebox{ and }\ \hG^B_{t,T}:=\hF^W_T \vee \hF^B_{t,T}\,.$$ 
 Since $P^m$ is not a finite measure, we introduce the probability
 $$P^\pi :=\int_{\R^d}P^x \, \pi (x)\, dx ,$$
 where $\pi$ is the Gaussian density and then 
 for any $\theta\in \cA^\Theta$, we consider the probability  $\hP^\pi_\theta:=P^\pi\otimes P_\theta$ and the (infinite) measure $\hP_\theta:=P^m\otimes P_\theta$ on $\Omega$.  
 Then clearly $\hcP^\pi_1:=\{\hP^\pi_\theta:\ \theta\in\cA^\Theta\}$ is weakly compact and its closure is nothing but  $\hcP^\pi=\{ P^\pi\}\otimes{\cP}$.\\
 We also introduce: $\hcP=\{ P^m\}\otimes{\cP}$.\\
 Moreover, under each probability $\hP^\pi\in\hcP^\pi$,  $(M_t)_{t\in [0,T]}$ is a centered  $(\widehat{\cF}^W_t )_{t\in [0,T]}$-martingale,  $(\Bt_t)_{t\in [0,T]}$ a centered $(\widehat{\cF}^B_{T-t,T} )_{t\in [0,T]}$ martingale and these processes are independent.\\
 Now, this permits to extend in a natural way the capacity and the sublinear expectation . So, we define the capacity
$$\widehat{c}(A):=\sup_{\widehat{P}^\pi\in\hcP^\pi}\widehat{P}^\pi(A),\ \ A\in\cB(\Omega),$$
and the sublinear expectation
$${\hE}^\pi[\xi]:=\sup_{\widehat{P}^\pi\in\hcP^\pi}E_{\widehat{P}^\pi}[\xi],\quad \xi\in L^0 (\Omega).$$ 
And all the related notions introduced at the beginning of Section \ref{DefCap} are naturally extended to the product space. Moreover, it is easy to verify that $\Bt$ remains a $G$-Brownian motion under $\hE^\pi$ and, by applying Fubini's Theorem under each probability $\hP^\pi$, that a set $A\in\cB (\Omega)$ is $\widehat{c}$-polar if and only if for $P^\pi$-almost or equivalently  $P^m$-almost all $\omega^1\in \Omega^W$, the set $\{\omega^2\in\Omega^B;\ (\omega^1 ,\omega^2)\in A\}$ is $c$-polar.\\
For any $p\geq 1$ and any Hilbert space $H$ we set: $$\widehat{L}^{p,\pi}_G(\Omega;H)=L^p_G(\Omega^B;L^p(\Omega^W,P^\pi,\cF_T^W;H)),$$
let us remark that  $\widehat{L}^{p,\pi}_G(\Omega;H)$ is not the completion of $Lip (\Omega ;H)$, that is the reason why we introduce the notation $\widehat{L}$.\\
We also introduce the following sets for any $p\geq1$:
$$\widehat{L}^{p}_G(\Omega;H)=L^p_G(\Omega^B;L^p(\Omega^W,P^m,\cF_T^W;H)),$$
with the usual convention that if $H=\R$ we omit it from the notations.\\
We also introduce a nonlinear integral :
$$\forall \xi \in \widehat{L}^{1}_G(\Omega) ,\ \hE [\xi]:=\sup_{\widehat{P}^\in\hcP}E_{\widehat{P}}[\xi].$$
\begin{lemma} Let $p\geq 1$ and $H$ be a separable Hilbert space, then the space $\widehat{L}^{p}_G(\Omega;H)$ equipped 
with the norm 
$$\forall X\in \widehat{L}^{p}_G(\Omega;H),\ \  \|X\|_{\widehat{L}^{p}_G(\Omega;H)}:=\left( \hE [\|X\|_{H}^p ]\right)^{1/p}$$
is a Banach space continuously embedded in $\widehat{L}^{p,\pi}_G(\Omega;H)$.
\end{lemma}
\begin{proof}
    This is consequence of the following inequality easy to obtain by density:
   $$\forall X\in \widehat{L}^{p,\pi}_G(\Omega;H),  \ \ \|X\|_{\widehat{L}^{p,\pi}_G(\Omega;H)}\leq \displaystyle\frac{1}{(2\pi)^{\frac{1}{2dp}}} \|X\|_{\widehat{L}^{p}_G(\Omega;H)} .$$
\end{proof} 
\begin{remark}
    As a consequence, for example, each element in $\widehat{L}^{p}_G(\Omega;H)$ admits a $\widehat{c}$-quasi-continuous version.
\end{remark}

We now introduce spaces of processes we'll work with: 
\begin{itemize}
 \item $\widehat{M}_G^{2,0}\left([0,T], ({\hF}^B_{t,T})_{t\in [0,T]}\right)$, the set  of simple processes $\Xi$ such that  for a given partition $\pi_T=\{t_0,...,t_N\}$ of $[0,T]$,
$$\Xi_t(\omega^1 ,\omega^2):=\sum_{i=0}^{N-1}\xi_{i+1}(\omega^1 ,\omega^2)I_{(t_i,t_{i+1}]}(t),$$where
for each $i=0,1,...,N-1$, $\xi_{i+1}\in \widehat{L}^2_G\left (\Omega \right)$ is $\widehat{\cF}^B_{t_{i+1},T}$-measurable;
\item $\widehat{M}_G^{2}\left([0,T], (\widehat{\cF}^B_{t,T})_{t\in [0,T]}\right)$, the completion of $\widehat{M}_G^{2,0}\left([0,T], ({\hF}^B_{t,T})_{t\in [0,T]}\right)$ under the
norm$$\left\|\Xi\right\|_{\widehat{M}_G^{2}\left([0,T], ({\hF}^B_{t,T})_{t\in [0,T]}\right)}:=\bigg\{\widehat{\bbE}\Big[\int_0^T\left|\Xi_t\right|^2dt\Big]\bigg\}^{1/2};$$
    \item $\widehat{M}_G^{2,0}\left([0,T], ({\hF}_{t})_{t\in [0,T]}\right)$, the set  of simple processes $\Xi$ such that  for a given partition $\pi_T=\{t_0,...,t_N\}$ of $[0,T]$,
$$\Xi_t(\omega^1,\omega^2):=\sum_{i=0}^{N-1}\xi_{i+1}(\omega^1 ,\omega^2)I_{(t_i,t_{i+1}]}(t),$$where
for each $i=0,1,...,N-1$, $\xi_{i+1}\in \widehat{L}^2_G\left (\Omega \right)$ is $\cF^W_{t_i}\otimes \cF^B_{t_{i+1},T}$-measurable;
\item $\widehat{M}_G^{2}\left([0,T], ({\hF}_{t})_{t\in [0,T]}\right)$, the completion of $\widehat{M}_G^{2,0}\left([0,T], ({\hF}_{t})_{t\in [0,T]}\right)$ under the
norm$$\left\|\Xi\right\|_{\widehat{M}_G^{2}\left([0,T], ({\hF}_{t})_{t\in [0,T]}\right)}:=\bigg\{\widehat{\bbE}\Big[\int_0^T\left|\Xi_t\right|^2dt\Big]\bigg\}^{1/2};$$
\item $\widehat{S}^2_G \left([0,T], ({\hF}_{t})_{t\in [0,T]}\right)$, the closed sub-vector space  of continuous processes in $\widehat{M}_G^{2}\left([0,T], ({\hF}_{t})_{t\in [0,T]}\right)$ such that 
$$\widehat{\bbE}\Big[\sup_{t\in [0,T]}|\Xi_t|^2\Big]<+\infty,$$
equipped with the norm
$$\| \Xi\|_{\widehat{S}^2_G \left([0,T], ({\hF}_{t})_{t\in [ 0,T]}\right)}:=\widehat{\bbE}\Big[\sup_{t\in [0,T]}|\Xi_t|^2\Big]^{1/2}.$$
\end{itemize}
Keeping the same notations as in Section \ref{G-integral}, if  $\Xi:=(\Xi^1,\cdots,\Xi^l)$ belongs to \\$\left(\widehat{M}_G^{2,0}\left([0,T], (\widehat{\cF}_{t})_{t\in [0,T]}\right)\right)^l$ of the form 
$$\forall j\in\{1,\cdots,l\},\ \forall t\in[0,T],\ \Xi_t^j(\omega):=\sum_{i=0}^{N-1}\xi_{i+1}^j(\omega)I_{(t_i,t_{i+1}]}(t),$$ 
we define for any $t\in[0,T]$ 
$$\overleftarrow{I}^\Xi_t :=\int_t^T \Xi_s\cdot \Bf_s=\int_0^{T-t}\Xi_{T-s}\cdot d\Bt_s:=\sum_{j=1}^l\sum_{i=0}^{N-1}\xi_{i+1}^j(B^j_{t_i\vee t}-B^j_{t_{i+1}\vee t}).$$
Similarly to Proposition \ref{propsi}, with same proof we have:
\begin{proposition}\label{propsi2}
The mapping $$\overleftarrow{I}:\left(\widehat{M}_G^{2,0}\left([0,T], (\widehat{\cF}_{t})_{t\in [0,T]}\right)\right)^l\rightarrow \widehat{S}^2_G \left([0,T], (\widehat{\cF}_{t})_{t\in [ 0,T]}\right)$$ defined by $\overleftarrow{I} (\Xi )=\left( \overleftarrow{I}^\Xi_t \right)_{t\in [0,T]}$ is a continuous linear mapping and can therefore be extended to a continuous linear operator from $\left(\widehat{M}_G^{2}\left([0,T], (\widehat{\cF}_{t})_{t\in [0,T]}\right)\right)^l$ into \\ $\widehat{S}^2_G \left([0,T], (\widehat{\cF}_{t})_{t\in [0,T]}\right)$ that we still denote by $\overleftarrow{I}$. Moreover, the following properties hold for all  $\Xi\in \left(\widehat{M}_G^{2}\left([0,T], (\widehat{\cF}_{t})_{t\in [0,T]}\right)\right)^l$: 
\begin{enumerate}
\item $\forall t\in[0,T]$, $\widehat{\bbE}\big[\int_t^T\Xi_s\cdot\Bf_s\big]=0$;
\item $\forall t\in[0,T]$, $\widehat{\bbE}\big[\big |\overleftarrow{I}^\Xi_t\big |^2\big]\leq\bar{\sigma}^2\widehat{\bbE}\left[\int_t^T\left |\Xi_s\right |^2ds\right]$, where $\bar{\sigma}$ is defined in Remark \ref{Rqb};
\item $\overleftarrow{I}$ satisfies the Doob's inequality:
\begin{equation}\label{Doob2}
\widehat{\bbE}\Big[\sup_{t\in[0,T]}\big|\overleftarrow{I}^\Xi_t\big|^2\Big]\leq
4\bar{\sigma}^2\widehat{\bbE}\int_0^T\left |\Xi_t\right |^2 dt.
\end{equation}
\end{enumerate}
\end{proposition}

\subsection{Forward integral w.r.t. $M$}
To construct the forward stochastic integral with respect to the $d-$dimensional martingale $M$, we still begin with the simple processes. Let $\Xi:=(\Xi^1,\cdots,\Xi^d)$ be in $\left(M_G^{2,0}\left([0,T], (\widehat{\cF}_{t})_{t\in [0,T]}\right)\right)^d$ of the form
$$\forall j\in\{1,\cdots,d\},\ \forall t\in[0,T],\ \Xi_t^j(\omega):=\sum_{i=0}^{N-1}\xi_{i+1}^j(\omega)I_{(t_i,t_{i+1}]}(t),$$
with $\xi_{i+1}^j\in L^2_G\left (\Omega^B , \cF^B_{t_{i+1}};L^2 (\Omega^W ,\cF^W_{t_{i}},P^\pi)\right)$. 
We define for any $t\in[0,T]$
$${I}^\Xi_t :=\int_0^t \Xi_s\cdot dM_s:=\sum_{j=1}^d\sum_{i=0}^{N-1}\xi^j_{i+1}(\omega)(M^j_{t_{i+1}\wedge t}-M^j_{t_i\wedge t}).$$
It is clear that the process $({I}^\Xi_t)_{t\in [0,T]}$ belongs to $M_G^{2}\left([0,T], (\widehat{\cF}_{t})_{t\in [0,T]}\right)$ and as a consequence of properties of process $(M_t )_{t\geq 0}$ mentioned in Section \ref{ForwardP}, the following proposition holds:
\begin{proposition}
The mapping $$I:\left(\widehat{M}_G^{2}\left( [0,T], ({\hF}_{t})_{t\in [0,T]}\right)\right)^d\rightarrow \widehat{S}^2_G\left([0,T], ({\hF}_{t})_{t\in [0,T]}\right)$$ defined by $I (\Xi)= (I_t^\Xi )_{t\in [0,T]}$ is a continuous linear mapping and hence can be continuously extended to $I:\left(\widehat{M}^{2}_G\left([0,T], ({\hF}_{t})_{t\in [0,T]}\right)\right)^d\rightarrow \widehat{S}^2_G\left([0,T], ({\hF}_{t})_{t\in [0,T]}\right)$. Moreover, the following properties hold for any $\Xi \in \left(\widehat{M}^{2}_G\left([0,T], ({\hF}_{t})_{t\in [0,T]}\right)\right)^d$: 
\begin{enumerate}[leftmargin=0.7cm]
\item $\forall t\in[0,T]$, $\widehat{\bbE}\left[\int_0^t\Xi_s\cdot dM_s\right]=0$;
\item $\forall t\in[0,T]$, $\widehat{\bbE}\left[\big|\int_0^t\Xi_s\cdot dM_s\big|^2\right]=2\widehat{\bbE}\left[\int_0^t \Xi_s\cdot a(X_s)\cdot \Xi^*_s\, ds\right]$;
\item $\widehat{\bbE}\left[\sup_{0\leq t\leq T}\big|\int_0^t\Xi_s\cdot dM_s\big|^2\right]\leq  8\widehat{\bbE}\left[\int_0^T\Xi_s\cdot a(X_s)\cdot \Xi^*_s \, ds\right] \leq 16\Lambda \widehat{\bbE}\left[\int_0^T\big|\Xi_s\big|^2ds\right].$
\end{enumerate}
\end{proposition}
\begin{remark}
    As a consequence of equality 2. and inequality \eqref{IneqBracket}, we also have
      \begin{eqnarray}{\label{IneqBracketGeneral}}
2\lambda\widehat{\bbE}\bigg[\int_0^T\big|\Xi_s\big|^2ds\bigg]\leq\widehat{\bbE}\bigg[\bigg|\int_0^T\Xi_s\cdot dM_s\bigg|^2\bigg]\leq  2\Lambda \widehat{\bbE}\bigg[\int_0^T\big|\Xi_s\big|^2ds\bigg].
       \end{eqnarray}    
\end{remark}
\subsection{A Representation Theorem}
It is well-known that one of the main tools to prove existence of solutions to BSDEs is a representation theorem. This is also our first step.
\begin{proposition}\label{represen} Let $s\in [0,T]$ and $Y\in \hL^2_G (\Omega)$ be $\hG^B_{s,T}$-measurable, then there exists a unique process $Z$ in $L^2_G(\Omega^B;L^2(\Omega^W\times[0,T], P^m\otimes dt;\R^d))$, which is $(\widehat{\cF}^W_t\vee\widehat{\cF}^B_{s,T} )_{t\in [0,T]}$-adapted such that 
$$Y=\hE^{\hG^W_0}[Y]+\int_0^T Z_r\cdot dM_r, \ \makebox{q.s.}$$
where $\hE^{\hG^W_0}[Y]$ is the unique variable in $\widehat{L}^2_G (\Omega)$ which is $\widehat{\cG}^W_0 -$measurable and such that for any $P\in\cP$, if we set $\widehat{P}^\pi=P^\pi \otimes P$, 
$E_{\widehat{P}^\pi}[Y|\widehat{\cG}^W_0]=\hE^{\hG^W_0}[Y].$
\end{proposition}
\begin{proof} Let us prove existence. Assume first that $Y$ admits the following decomposition:
$$Y=\sum_{j=1}^n S^jT^j,$$
where $n\geq 1$ and for each $j\in \{1,\cdots n\}$, $S^j$ is bounded and belongs to $ L^2 (\Omega^W,P^m)$ and $T^j$ belongs to $ L^2_G (\Omega^B)$ and is $\widehat{\cF}^B_{s,T}$-measurable. Then by \cite[Theorem 4.7]{BPS} we know that a representation theorem holds. More precisely, there exists a properly exceptional set $\cN\subset{\R^d}$ (see \cite[Section 4.1]{FOT}) hence $m$-negligeable, independent of the variables $(S^j)_j$  and an $\R^d$-valued $(\widehat{\cF}^W_t)_{t\in [0,T]}$-predictable process $Z^j$ such that for all $x\in\R^d \setminus \cN$, 
$$S^j =E_{P^x}[S^j | \cF^W_0]+\int_0^T Z^j_r\cdot dM_r,\ P^x-\makebox{a.s}.$$
Moreover, we have the following estimate:
$$\lambda E_{P^x} \bigg[\int_0^T |Z_s |^2  ds\bigg]\leq  E_{P^x} \bigg[\int_0^T |Z_s\cdot \sigma (X_s) |^2  ds\bigg]\leq\displaystyle\frac12 E_{P^x}\big[|S^j |^2\big]. $$
This yields:
$$ E_{P^m} \bigg[\int_0^T |Z_s |^2  ds\bigg]=\int_{\R^d} E_{P^x} \bigg[\int_0^T |Z_s |^2  ds\bigg]dx\leq \displaystyle\frac{1}{2\lambda} E_{P^m}\big[|S^j |^2\big],$$
hence $Z^j$ belongs to  
$L^2 (\Omega^W\times [0,T], P^m\otimes dt;\R^d)$ and $\int_0^T Z^j_s dM_s$  to $L^2 (\Omega^W, \cF_T^W ,P^m)$. 
Still following \cite{BPS}, we know that we also have
$$S^j =E_{P^\pi}[S^j | \cF^W_0]+\int_0^T Z^j_r\cdot dM_r,\ P^\pi-\makebox{a.s}.$$
Since $P^\pi$ and $P^m$ are equivalent measures this yields
$$S^j =E_{P^\pi}[S^j | \cF^W_0]+\int_0^T Z^j_r\cdot dM_r,\ P^m-\makebox{a.s},$$
and since $E_{P^\pi}[S^j | \cF^W_0]=S^j-\int_0^T Z^j_r\cdot dM_r$, it belongs to $L^2 (\Omega^W,\cF^W_0, P^m )$.

Now we put  $Z_t:=\sum_{j=1}^n Z^j_t  T^j$ which is clearly $(\widehat{\cF}^W_t\vee\widehat{\cF}^B_{s,T} )_{t\in [0,T]}$-adapted and we define:
$$\hE^{\hG^W_0}[ Y]:=\sum_{j=1}^n T^jE_{P^\pi}[S^j | \cF^W_0],$$
and remark that for any $P\in \cP$, if we set $\widehat{P}^\pi=P^\pi\otimes P$ and $\widehat{P} =P^m\otimes P$,

$$\hE^{\hG^W_0}[Y]=E_{\widehat{P}^\pi}[Y|\hG^W_0]\makebox{ and }Y=\hE^{\hG^W_0}[Y]+\int_0^T Z_r\cdot dM_r,\ \widehat{P}^m-\makebox{a.s}.$$
Then we have:
\begin{equation*}\begin{split}
E_{\widehat{P}}\left[\left|\hE^{\hG^W_0}[Y]\right|^2\right]&=\int_{\R^d}E_{P}\left[E_{P^x}\left[\left|\hE^{\hG^W_0}[Y]\right|^2\right]\right]dx\\
&= \int_{\R^d}E_{P}\bigg[E_{P^x}\bigg[\Big|Y-\int_0^T Z_s\cdot dM_s\Big|^2\bigg]\bigg]dx\\
&= \int_{\R^d}E_{P}\left[E_{P^x}\left[\big|E_{P^x}[Y|\cF^W_0]\big|^2\right]\right]dx\\
&\leq \int_{\R^d}E_{P}\left[E_{P^x}\left[E_{P^x}\big[|Y|^2|\cF^W_0\big]\right]\right]dx\\
&=E_{\widehat{P}}\big[|Y|^2\big].
\end{split}\end{equation*}
Taking supremum over $P\in\cP$ we get:
$$\left\|\hE^{\hG^W_0}[Y]\right\|_{\widehat{L}^2_G (\Omega )}\leq \| Y\|_{\widehat{L}^2_G (\Omega)}.$$
It is easy to conclude thanks to a density argument. Uniqueness is a consequence of the inequality above.
\end{proof} 
\begin{corollary}
Let $t\in [0,T]$ and for any $Y\in \widehat{L}^2_G (\Omega ,\hG^W_{T} )$ which admits the representation 
$$Y=\hE^{\hG^W_0}[Y]+\int_0^T Z_r\cdot dM_r, \ \makebox{q.s.},$$ with $Z\in M^2_G([0,T],(\hG^W_s)_{s\in [0,T]};\R^k)$, we put 
in a natural way 
$$\hE^{\hG^W_t}[Y]=\hE^{\hG^W_0}[Y]+\int_0^t Z_r\cdot dM_r.$$
Then the map $Y\mapsto \hE^{\hG^W_t}[Y]$ defines a linear continuous map from $\widehat{L}^2_G (\Omega ,\hG^W_{T} )$ into $\widehat{L}^2_G (\Omega ,\hG^W_{t} )$.
\end{corollary}
\subsection{The existence and uniqueness of solution to BDSDEs}
Let $\xi\in\widehat{L}^2_G(\Omega,\widehat{\cal{F}}_T)$ and
\begin{eqnarray*}
f:\Omega\times[0,T]\times\R\times\R^d&\longrightarrow&\R\,,\\
g:\Omega\times[0,T]\times\R\times\R^d&\longrightarrow&\R^l,
\end{eqnarray*}
fulfilling the following assumptions:\\

{\bf Hypotheses (H):}
{\it \begin{enumerate}[leftmargin=*]
    \item If $U$ and $V$ are processes belonging respectively to $\hM^2_G([0,T], (\hF_t)_{t\in [0,T]})$ and \\
    $\left(\hM^2_G  ([0,T], (\hF_t)_{t\in [0,T]})\right)^{ d}$ then $$(t,\omega)\in [0,T]\times \Omega\longrightarrow f(t,\omega,U_t(\omega),V_t(\omega))$$ and $$(t,\omega)\in [0,T]\times \Omega\longrightarrow g(t,\omega,U_t (\omega), V_t (\omega))$$  belong to $\hM^2_G  ([0,T], (\hF_t)_{t\in [0,T]})$ and $\left(\hM^2_G  ([0,T], (\hF_t)_{t\in [0,T]})\right)^l$, respectively.
    \item There exist non-negative constants $K$ and $\alpha$ such that for all $\omega,t,y,y',z,z'$ in $\Omega\times[0,T]\times \R\times \R\times\R^d\times\R^d$:
\begin{equation*}
\begin{split}
|f(\omega,t,y,z)-f(\omega,t,y',z')|^2&\leq K\left( |y-y'|^2 +|z-z'|^2\right);\\
|g(\omega,t,y,z)-g(\omega,t,y',z')|^2&\leq K|y-y'|^2 +\alpha |z-z'|^2.
\end{split}
\end{equation*}
\item The contraction property: $\alpha\Lambda\bar{\sigma}^2<2\lambda$.
\end{enumerate}}
We consider the following  doubly stochastic backward equation:
\begin{equation}\label{eqBDSDEs}
Y_t=\xi +\int_t^T f(s,Y_s,Z_s\sigma(X_s))ds +\int_t^T g(s,Y_s ,Z_s\sigma(X_s))\cdot\Bf_s-\int_t^T Z_s\cdot dM_s.
\end{equation}
\begin{theorem}{\label{MainTheo}} Under {\bf(H)}, equation \eqref{eqBDSDEs} admits a unique solution
$(Y,Z)$  belonging to  $\hS^2_G ([0,T], (\hF_t)_{t\in [0,T]})\times \left(\hM^2_G ([0,T], (\hF_t)_{t\in [0,T]})\right)^d.$
\end{theorem}
We split the proof in several steps:
\subsubsection{Step one: the linear case}  
In this subsection, we assume that coefficients $f$ and $g$ do not depend on $(y,z)$, in other words that they 
belong to $\hM^2_G  ([0,T], (\hF_t)_{t\in [0,T]})$ and $\left(\hM^2_G  ([0,T], (\hF_t)_{t\in [0,T]})\right)^l$ respectively. \\
Let us first remark that if it exists, solution is unique.\\ Indeed, let $(\hat{Y},\hat{Z})\in\hS^2_G  ([0,T], (\hF_t)_{t\in [0,T]})\times\left(\hM^2_G  ([0,T], (\hF_t)_{t\in [0,T]})\right)^d$ denote the difference of two solutions, then it is obvious that 
\begin{equation*}
\hat{Y}_t+\int_t^T\hat{Z}_s\cdot dM_s=0,\quad t\in[0,T].
\end{equation*}
Hence, under each $\widehat{P}\in\widehat{\cP}$, by It\^o's formula we have
\begin{eqnarray*}
    0=\hY_0^2 +2\int_0^T \hY_s \hZ_s\cdot dM_s+2\int_0^T \hZ_s \cdot a(X_s)\cdot \hZ_s^* ds.
\end{eqnarray*}
Applying Burkholder-Davis-Gundy inequality under each probability $P^x$, we get that there exists a universal constant $c_1$ such that
\begin{equation*}\begin{split}
E_{\widehat{P}} \bigg[ \sup_{t\in [0,T]}\bigg|\int_0^t \hY_s \hZ_s\cdot dM_s\bigg|\bigg]&\leq c_1
E_{\widehat{P}} \bigg[\Big(2\int_0^T |\hY_s|^2 \hZ_s \cdot a(X_s)\cdot \hZ_s^* ds\Big)^{1/2}\bigg]\\ &\leq  c_1 \sqrt{2\Lambda} E_{\widehat{P}} \bigg[\Big( \int_0^T |\hY_s|^2 |\hZ_s|^2 ds\Big)^{1/2}\bigg]\\
&\leq  c_1\sqrt{2\Lambda}  E_{\widehat{P}} \bigg[\sup_{s\in [0,T]}|\hY_s |^2 +\int_0^T  |\hZ_s|^2 ds\bigg]<+\infty,
\end{split}\end{equation*}
so the ``martingale part" vanishes when integrating under $\widehat{P}$, this yields:
\begin{equation*}\begin{split}
    0&=E_{\widehat{P}} [\hY_0^2] + 2E_{\widehat{P}} \bigg[\int_0^T \hZ_s \cdot a(X_s)\cdot \hZ_s^* ds\bigg]\\
    &\geq  E_{\widehat{P}} [\hY_0^2] + 2\lambda E_{\widehat{P}}\bigg[\int_0^T |\hZ_s |^2 ds\bigg]
\end{split}\end{equation*}
which provides  $\hat{Z}_t=0$    $dtdP-$a.e. for any $\widehat{P}\in\widehat{\cP}$, this ensures that $\widehat{c}$-quasi-surely, $Z_t =0$ $dt$-a.e. hence for all $t\in [0,T]$, $\hY_t =0$ $\widehat{c}$-quasi-surely,  which proves uniqueness.\\
Let us now study existence of the solution.\\
We are given $0\leq u<v\leq T$ and assume first that $f$ and $g^j$,for any $j\in \{ 1,\cdots ,l\}$,  
are in the elementary form: $\forall s\in [0,T],\ \forall \omega:=(\omega^1,\omega^2 )\in\Omega$
$$f (s,\omega)=S^f (\omega^1)T^f (\omega^2)I_{(u,v]}(s)\makebox{ and }g^j (s,\omega)=S^{j,g} (\omega^1)T^{j,g} (\omega^2)I_{(u,v]}(s),$$
where $S^f, S^{j,f}$ belong to $L^2 (\Omega^W, \cF^W_u ,P^\pi )$ and $T^g, T^{j,g}$ to $L^2_G (\Omega^B )$ and are 
$\cF^B_{v,T}$-measurable. \\
The equation becomes
\begin{equation}\label{eqsimple}
Y_t =\xi +\int_t^T f(s)ds-\int_t^T Z_s\cdot dM_s +\int_t^T g(s)\cdot\Bf_s .\end{equation}
We define for all $t\in [0,T]$:
\begin{equation*}\begin{split}
\cM_t &=\hE^{\hG^W_t}\Big[\xi +\int_0^Tf(s)ds+\int_0^T g(s)\cdot\Bf_s\Big]\\
&=\hE^{\hG^W_t}\Big[\xi +S^fT^f (v-u)+\sum_{j=1}^l S^{j,g}T^{j,g}(B_v^j -B_u^j )\Big].
\end{split}\end{equation*}
Then by Proposition \ref{represen} we know that there exist  $({\cF}^W_t)_{t\in [0,T]}$-predictable processes $Z^\xi$, $Z^f$ and $\left( Z^{j,g}\right)_{1\leq j\leq l}$  in 
$L^2 (\Omega^W\times [0,T];\R^d)$ such that $$\xi= \hE^{\hG^W_0}[\xi] +\int_0^T Z^\xi_s \cdot dM_s\ ,\ S^f = \hE^{\hG^W_0}[S^f]+\int_0^u Z^f_s\cdot dM_s\ ,\ 
$$
and $$ S^{j,g}=\hE^{\hG^W_0}[S^{j,g}]+\int_0^u Z^{j,g}_s \cdot dM^j_s.$$
This yields for all $t\in [0,T]$:
\begin{equation*}\begin{split}\cM_t=&\,\hE^{\hG^W_0}[\xi]+ \hE^{\hG^W_0}[S^f]T^f(v-u)+\sum_{j=1}^l \hE^{\hG^W_0}[S^{j,g}]T^{j,g}(B^j_v -B^j_u)\\&+\int_0^{t} \Big( Z^\xi_s +Z^f_s T^f (v-u)I_{[0,u]}(s)+\sum_{j=1}^l Z^{j,g}_sT^{j,g}(B_v^j -B_u^j )I_{[0,u]}(s)\Big)\cdot dM_s\\
=&\, \hE^{\hG^W_0}[\xi]+\hE^{\hG^W_0}[S^f]T^f+\sum_{j=1}^l \hE^{\hG^W_0}[S^{j,g}]T^{j,g}+\int_0^t Z_r\cdot dM_r,
\end{split}\end{equation*}
where we set: $\forall s\in [0,T],\ \forall \omega:=(\omega^1,\omega^2)\in \Omega :$
\begin{equation*}
    \begin{split}
Z_s (\omega)=&Z^\xi_s (\omega^1 )+\\&\bigg(Z^f_s (\omega^1) T^f (\omega^2) (v-u)+
\sum_{j=1}^l Z^{j,g}_s (\omega^1 )T^{j,g}(\omega^2)(B_v^j -B_u^j )(\omega^2)\bigg) I_{[0,u]}(s).\end{split}\end{equation*}
Clearly the process $(Z_t )_{t\in [0,T]}$ is $(\hF_t )_{t\in [0,T]}$-adapted. Moreover, we have the following estimates:
\begin{equation*}\begin{split}
    \hE \left[\int_0^u |Z^f_s  T^f  (v-u)|^2ds\right]&=E_{P^m}\bigg[\int_0^u |Z^f_s |^2 ds\bigg]\mathbb{E} \big[|T^f|^2\big](v-u)^2,\\
    \hE \left[ \int_0^u |Z^{j,g}_s T^{j,g}(B_v^j -B_u^j )|^2ds\right]&\leq E_{P^m}\bigg[\int_0^u |Z^{j,g}_s |^2 ds\bigg]\mathbb{E} \big[|T^{j,g}|^2\big]2\Lambda (v-u).
\end{split}\end{equation*}
Now, by considering sequences of elementary processes approaching $Z^f$ and the $Z^{j,g}$, thanks to these estimates we clearly get that $Z$ belongs to $\left(\hM^2_G  ([0,T], (\hF_t)_{t\in [0,T]})\right)^d$.\\
Now, we put for all $t\in [0,T]$:
$$Y_t =\hE^{\hG^W_t}\bigg[\xi +\int_t^Tf(s)ds+\int_t^T g(s)\cdot\Bf_s\bigg],$$ 
so that by considering $\cM_T -\cM_t$ we have:
$$Y_t =\xi +\int_t^Tf(s)ds+\int_t^T g(s)\cdot\Bf_s-\int_t^T Z_s \cdot dM_s.$$
We know  that $Y_t$ is { $\hG^W_t$}-measurable, the above relation ensures that $Y_t $ is $\hF^B_{t,T}$-measurable.
From this, we get that $(Y_t )_{t\in [0,T]}$ is $(\widehat{\mathcal{F}}_t )_{t\in [0,T]}$-adapted and belongs to $\hS^2_G  ([0,T], (\hF_t)_{t\in [0,T]})$.
Hence $(Y,Z)$ solves \eqref{eqsimple} in that case.\\
Then, consider the case where coefficients $f$ and $g$ are finite combinations of functions as the one we have just considered. More precisely, assume that for a given partition $\pi_T=\{t_0,...,t_N\}$ of $[0,T]$, 
\begin{equation}\label{eqf}
    \forall s\in [0,T],\ f_s =\sum_{i=0}^{N-1}(\sum_{k=1}^n S^f_{i,k}T^f_{i,k})I_{(t_i ,t_{i+1}]}(s),
\end{equation}
and for all $j\in \{1,\cdots ,l\}$
\begin{equation}\label{eqg}
 \forall s\in [0,T],\ g^j_s =\sum_{i=0}^{N-1}(\sum_{k=1}^n S^{j,g}_{i,k}T^{j,g}_{i,k})I_{(t_i ,t_{i+1}]}(s).   
\end{equation}
Then by linearity, we construct the solution to equation \eqref{eqsimple} as linear combinations of solutions we have constructed before.\\
To end the proof of this first step, we shall need the following It\^o's formula. 
\begin{theorem}\label{Itoformula} 
Let $Y\in \widehat{S}^2_G([0,T], (\hF_t)_{t\in [0,T]})$ such that there exist \\$b\in \widehat{M}^2_G ([0,T], (\hF_t)_{t\in [0,T]}; \R)$, $\beta\in \left( \widehat{M}^2_G ([0,T], (\hF_t)_{t\in [0,T]})\right)^l$ and \\$\gamma\in \left( \widehat{M}^2_G ([0,T], (\hF_t)_{t\in [0,T]})\right)^d$ satisfying:
\begin{equation*}\label{forwardbackwardsde}
Y_t=Y_0+\int_0^tb_sds+\int_0^t\beta_s\cdot\Bf_s+\int_0^t\gamma_s\cdot dM_s,\quad 0\leq t\leq T.
\end{equation*}
Then 
\begin{equation}\label{itoformulaequationcarre}\begin{split}
|Y_t|^2=&\,|Y_0|^2+2\int_0^tY_sb_sds+2\int_0^tY_s\beta_s\cdot\Bf_s+2\int_0^tY_s\gamma_s\cdot dM_s\\&- \int_0^t \sum_{1\leq i,j\leq l}\beta^i_s \beta^j_sd\langle \tilde{B}^i,\tilde{B}^j \rangle_s+ 2\int_0^t\gamma_s\cdot a(X_s)\cdot\gamma_s^*\,ds,\quad \widehat{c}-q.s.
\end{split}\end{equation}
More generally, let $\varphi:\R^+\times\R\rightarrow\R$ be functions of class $C^{1,2}$. Denote by $\varphi'$ and $\varphi''$ the derivatives of $\varphi$ w.r.t. the space variable and $\frac{\partial\varphi}{\partial t}$ the derivative w.r.t. time. We assume that $\varphi''$ is bounded. Then, it holds that
\begin{equation}\label{itoformulaequation}\begin{split}
\varphi(t,Y_t)=&\,\varphi_0(Y_0)+\int_0^t\frac{\partial\varphi}{\partial s}(s,Y_s)ds+\int_0^t\varphi_s'(Y_s)b_sds+\int_0^t\varphi_s'(Y_s)\beta_s\cdot\Bf_s\\&+\int_0^t\varphi_s'(Y_s)\gamma_s\cdot dM_s-\frac{1}{2}\int_0^t \sum_{1\leq i,j\leq l}\varphi_s''(Y_s)\beta^i_s\beta^j_sd\langle \tilde{B}^i,\tilde{B}^j \rangle_s\\&+\int_0^t\int_0^t\varphi_s''(Y_s)\gamma_s\cdot a(X_s)\cdot\gamma_s^*\, ds,\quad \widehat{c}-q.s.
\end{split}\end{equation}
Let $\tilde{b},\tilde{\beta},\tilde{\gamma}$ satisfying the same conditions as $b,\beta,\gamma$ and $\tilde{Y}$ admitting the following expression
\begin{equation*}\label{forwardbackwardsdebis}
\tilde{Y}_t=\tilde{Y}_0+\int_0^t\tilde{b}_sds+\int_0^t\tilde{\beta}_s\cdot\Bf_s+\int_0^t\tilde{\gamma}_s\cdot dM_s,\quad 0\leq t\leq T.
\end{equation*}
Then the following It\^o product rule holds,
\begin{equation}\label{itoproductrule}\begin{split}
Y_t\tilde{Y}_t=&\,Y_0\tilde{Y}_0+\int_0^tY_sd\tilde{Y}_s+\int_0^t\tilde{Y}_s dY_s- \int_0^t \sum_{1\leq i,j\leq l} \displaystyle\frac{\beta^i_s \tilde{\beta}^j_s +\beta^j_s \tilde{\beta}^i_s}{2}d\langle \tilde{B}^i,\tilde{B}^j \rangle_s\\& +\int_0^t\left( \gamma_s\cdot a(X_s)\cdot\tilde{\gamma}_s^* +\tilde{\gamma}_s\cdot a(X_s)\cdot{\gamma}_s^* \right)\,ds, \quad \widehat{c}-q.s.
\end{split}\end{equation}
\end{theorem}
\begin{proof} Under each $\hP^\pi \in\widehat{\cP}^\pi$, the same proof as the one of \cite[Lemma 1.3]{PardouxPeng94} applies. Then since each term in \eqref{itoformulaequation} admits a quasi-continuous version, equality holds quasi-surely.

\end{proof}

Consider now $f\in \hM^2_G  ([0,T], (\hF_t)_{t\in [0,T]})$ and $g\in\left(\hM^2_G  ([0,T], (\hF_t)_{t\in [0,T]})\right)^l$, it is clear that there exist sequences $(f^n)_{n\in\N}$ and $(g^{n,j})_{n\in\N}$  of the form \eqref{eqf} and \eqref{eqg} converging respectively to $f$ and $g^j$ in $\hM^2_G([0,T],(\hF_t)_{t\in [0,T]})$ for all $j\in \{ 1,\cdots ,l\}$.\\
For each $n\in\N$, we denote by $(Y^n, Z^n )$ the solution of 
\begin{equation}\label{eqapx}
Y^n_t =\xi +\int_t^T f^n(s)ds-\int_t^T Z^n_s\cdot dM_s +\int_t^T g^n(s)\cdot\Bf_s .\end{equation}
Let $n,m\in\N$, for any process $h$ we set $h^{n,m}:=h^n -h^m$ so that $(Y^{n,m}, Z^{n,m})$ is the solution of 
$$\forall t\in [0,T], \ Y^{n,m}_t =\int_t^T f^{n,m}(s)ds-\int_t^T Z^{n,m}_s\cdot dM_s +\int_t^T g^{n,m}(s)\cdot\Bf_s.$$
It\^o's formula yields for any $t\in [0,T]$:
\begin{align*}
 |Y^{n,m}_t|^2=&\,2\int_t^T Y^{n,m}_sf^{n,m}_s ds +2\int_t^T Y^{n,m}_sg^{n,m}_s\cdot\Bf_s-2\int_t^T Y^{n,m}_s Z^{n,m}_s\cdot dM_s\\
 &+\int_t^T \sum_{1\leq i,j\leq l }g^{n,m,i}_s g^{n,m,j}_s d\langle \tilde{B}^i ,\tilde{B}^j \rangle_s-2\int_t^T  Z^{n,m}_s\cdot a(X_s)\cdot (Z^{n,m}_s)^*ds\\
 \leq&\,2\int_t^T |Y^{n,m}_s||f^{n,m}_s| ds +2\int_t^T Y^{n,m}_sg^{n,m}_s\cdot\Bf_s-2\int_t^T Y^{n,m}_s Z^{n,m}_s\cdot dM_s\\
 &+\bar{\sigma}^2 \int_t^T |g^{n,m}_s|^2 ds -2\lambda \int_t^T | Z^{n,m}_s|^2 ds.
\end{align*}
So we have
\begin{equation}\label{cvz}\begin{split}
|Y^{n,m}_t|^2+2\lambda \int_t^T | Z^{n,m}_s|^2 ds\leq &\,2\int_t^T |Y^{n,m}_s| |f^{n,m}_s| ds +2\int_t^T Y^{n,m}_sg^{n,m}_s\cdot\Bf_s\\&-2\int_t^T Y^{n,m}_s Z^{n,m}_s\cdot dM_s+\bar{\sigma}^2 \int_t^T |g^{n,m}_s|^2 ds. 
\end{split}\end{equation}
Then under each $\hP=P^m\otimes P\in\widehat{\mathcal{P}}$, it holds 
\begin{align*}
&E_{\widehat{P}} \bigg[ \Big( \int_t^T |Y^{n,m}_s|^2\sum_{1\leq i,j\leq l}g^{n,m,i}_s g^{n,m,j}_s d\langle \tilde{B}^i ,\tilde{B}^j \rangle_s \Big)^{1/2}\bigg]\\=&\,\int_{\R^d}E_{P^x\otimes P} \bigg[ \Big( \int_t^T |Y^{n,m}_s|^2\sum_{1\leq i,j\leq l}g^{n,m,i}_s g^{n,m,j}_s d\langle \tilde{B}^i ,\tilde{B}^j \rangle_s \Big)^{1/2}\bigg]dx\\\leq&\, E_{\widehat{P}} \bigg[ \Big( \bar{\sigma}^2 \int_t^T|Y^{n,m}_s|^2|g^{n,m}_s|^2 ds\Big)^{1/2}\bigg]\\\leq&\, \bar{\sigma}\bigg( E_{\widehat{P}} \Big[ \sup_{s\in [t,T]}|Y^{n,m}_s |^2\Big]+ E_{\widehat{P}} \Big[\int_t^T|g^{n,m}_s|^2 ds\Big]\bigg)<+\infty,
\end{align*}
hence $E_{\widehat{P}} \big[\int_t^T Y^{n,m}_sg^{n,m}_s\cdot\Bf_s\big]=0 $. Similarly
$$E_{\widehat{P}} \bigg[\int_t^T Y^{n,m}_s Z^{n,m}_s\cdot dM_s\bigg]=E_{\widehat{P}} \bigg[\int_0^T Y^{n,m}_s Z^{n,m}_s\cdot dM_s-\int_0^t Y^{n,m}_s Z^{n,m}_s\cdot dM_s\bigg]=0,$$
as a consequence we get
\begin{align*}
  E_{\widehat{P}} &\big[|Y^{n,m}_t|^2\big]+2\lambda E_{\widehat{P}} \bigg[\int_t^T | Z^{n,m}_s|^2 ds\bigg]\\\leq & 2E_{\widehat{P}} \bigg[\int_t^T |Y^{n,m}_s| |f^{n,m}_s| ds\bigg] +\bar{\sigma}^2E_{\widehat{P}} \bigg[ \int_t^T |g^{n,m}_s|^2 ds\bigg]  \\
  \leq &E_{\widehat{P}} \bigg[\int_t^T |Y^{n,m}_s|^2 ds\bigg]+E_{\widehat{P}} \bigg[\int_t^T |f^{n,m}_s|^2 ds\bigg] +\bar{\sigma}^2E_{\widehat{P}} \bigg[ \int_t^T |g^{n,m}_s|^2 ds\bigg]\\
  \leq & E_{\widehat{P}} \bigg[\int_t^T |Y^{n,m}_s|^2 ds\bigg]+\| f^{n,m}\|^2_{\hM^2_G  ([0,T], (\hF_t)_{t\in [0,T]})}+\bar{\sigma}^2\| g^{n,m}\|^2_{(\hM^2_G  ([0,T], (\hF_t)_{t\in [0,T]}))^l}.
\end{align*}
As a consequence of Gronwall's lemma, $\forall t\in [0,T],\ \forall \widehat{P}\in \widehat{\mathcal{P}}:$
$$\ E_{\widehat{P}} \big[|Y^{n,m}_t|^2\big]\leq \Big(\| f^{n,m}\|^2_{\hM^2_G  ([0,T], (\hF_t)_{t\in [0,T]})}+\bar{\sigma}^2\|g^{n,m}\|^2_{(\hM^2_G  ([0,T], (\hF_t)_{t\in [0,T]}))^l}\Big)e^T.$$
It is easy to conclude that $(Y^n )_n$ is a Cauchy sequence in $\hM^2_G  ([0,T], (\hF_t)_{t\in [0,T]})$ hence converges to a process $Y\in \hM^2_G  ([0,T], (\hF_t)_{t\in [0,T]})$. Coming back to equation \eqref{cvz},by taking expectation under each $\widehat{P}\in \widehat{\mathcal{P}}$ and then the supremum, we see that $(Z^n)_n$ is also a Cauchy sequence hence converges to a process $Z$ in $(\hM^2_G  ([0,T], (\hF_t)_{t\in [0,T]}))^l$. Passing to the limit in \eqref{eqapx}, we get that $(Y,Z)$ satisfies our equation:
\begin{equation}\label{limiteqapx}
Y_t =\xi +\int_t^T f(s)ds-\int_t^T Z_s\cdot dM_s +\int_t^T g(s)\cdot\Bf_s ,\end{equation}
from this the fact that $Y$ belongs to $\hS^2_G  ([0,T], (\hF_t)_{t\in [0,T]})$ is clear, which achieves the proof in the linear case.

\subsubsection{Step two: the general case}

We now assume that  $f$ and $g$ satisfy hypothese {\bf (H)}.  Let us consider the Picard sequence $\{(Y^n,Z^n)\}_n$ defined by $(Y^0,Z^0)=(0,0)$ and for all $n\in\bbN$, we denote by $(Y^{n+1},Z^{n+1})\in \hM^2_G  ([0,T], (\hF_t)_{t\in [0,T]})\times\left(\hM^2_G  ([0,T], (\hF_t)_{t\in [0,T]})\right)^d$ the solution of the linear GBDSDE
\begin{equation*}\label{picardequation}
Y_t^{n+1} =\xi +\int_t^Tf(s,Y^n_s,Z^n_s\sigma(X_s))ds+\int_t^T g(s,Y_s^n,Z_s^n\sigma(X_s))\cdot\Bf_s-\int_t^T Z_s^{n+1}\cdot dM_s.
\end{equation*}
Let us remark that process $(\sigma (X_s ))_{s\in [0,T]}$ is limit in $L^2 (\Omega^W\times [0,T], P^W\otimes dt)$ of simple processes, see for example \cite[Ch.3-Proposition 2.8]{KarSh}, from this it is easy to get that $(Z_s \sigma (X_s ))_{s\in [0,T]}$ belongs to $\left(\hM^2_G  ([0,T], (\hF_t)_{t\in [0,T]})\right)^d$ and  then that  \\$(f(s,Y^n_s,Z^n_s\sigma(X_s)))_{s\in [0,T]}$ and $(g(s,Y_s^n,Z_s^n\sigma(X_s)))_{s\in [0,T]}$ belong respectively to \\$\hM^2_G  ([0,T], (\hF_t)_{t\in [0,T]})$ and $\left(\hM^2_G  ([0,T], (\hF_t)_{t\in [0,T]})\right)^d$.\\
Let $\hat{Y}^{n+1}:=Y^{n+1}-Y^n$ and $\hat{Z}^{n+1}:=Z^{n+1}-Z^n$. Then applying It\^o's formula to $e^{\beta t}(\hat{Y}^{n+1}_t)^2$ with given constant $\beta>0$, we have   
\begin{equation*}
\begin{split}
&|\hat{Y}^{n+1}_{0}|^2+\beta\int_{0}^Te^{\beta s}|\hat{Y}^{n+1}_s|^2ds+2\int_{0}^Te^{\beta s}\hat{Z}^{n+1}_s\cdot a(X_s)\cdot (\hat{Z}^{n+1}_s)^* ds\\=&\,2\int_{0}^Te^{\beta s}\hat{Y}^{n+1}_s(f_s(Y_s^n,Z_s^n\sigma(X_s))-f_s(Y^{n-1}_s,Z^{n-1}_s\sigma(X_s)))ds\\+&2\int_{0}^Te^{\beta s}(g_s(Y_s^n,Z_s^n\sigma(X_s))-g_s(Y_s^{n-1},Z_s^{n-1}\sigma(X_s)))\cdot\Bf_s \\&-2\int_{0}^T e^{\beta s} \hat{Y}^{n+1}_s \hat{Z}^{n+1}_s\cdot dM_s \\
&+\int_{0}^Te^{\beta s}\sum_{1\leq i,j\leq l }(g^i_s(Y_s^n,Z_s^n\sigma(X_s))-g^i_s(Y_s^{n-1},Z_s^{n-1}\sigma(X_s)))\\&\hspace{3cm}(g^j_s(Y_s^n,Z_s^n\sigma(X_s))-g^j_s(Y_s^{n-1},Z_s^{n-1}\sigma(X_s)))d\langle \tilde{B}^i,\tilde{B}^j\rangle_s.
\end{split}
\end{equation*}
Under each $\widehat{P}\in\widehat{\cP}$, by using Cauchy-Schwartz inequality and Lipschitz conditions in hypotheses {\bf (H)}, we obtain
\begin{equation*}
\begin{split}
&2E_{\widehat{P}}\int_{0}^Te^{\beta s}\hat{Y}^{n+1}_s(f_s(Y_s^n,Z_s^n\sigma(X_s)))-f_s(Y^{n-1}_s,Z^{n-1}_s\sigma(X_s))))ds\\\leq&\, \frac{1}{\epsilon}E_{\widehat{P}}\int_{0}^Te^{\beta s}|\hat{Y}^{n+1}|^2ds+K\epsilon E_{\widehat{P}}\int_{0}^Te^{\beta s}|Y^n_s-Y^{n-1}_s|^2ds\\&+K\epsilon E_{\widehat{P}}\int_{0}^Te^{\beta s}|(Z^n_s-Z^{n-1}_s)\sigma(X_s))|^2ds
\end{split}
\end{equation*}
and
\begin{equation*}
\begin{split}
&E_{\widehat{P}}\int_{0}^Te^{\beta s}\sum_{1\leq i,j\leq l }(g^i_s(Y_s^n,Z_s^n\sigma(X_s))-g^i_s(Y_s^{n-1},Z_s^{n-1}\sigma(X_s)))\\&\hspace{3cm}(g^j_s(Y_s^n,Z_s^n\sigma(X_s))-g^j_s(Y_s^{n-1},Z_s^{n-1}\sigma(X_s)))d\langle \tilde{B}^i,\tilde{B}^j\rangle_s.\\
\leq &E_{\widehat{P}}\int_{0}^Te^{\beta s}\bar{\sigma}^2 |g_s(Y_s^n,Z_s^n\sigma(X_s))-g_s(Y_s^{n-1},Z_s^{n-1}\sigma(X_s))|^2ds\\
\leq& K\bar{\sigma}^2E_{\widehat{P}}\int_{0}^Te^{\beta s}|Y^n_s-Y^{n-1}_s|^2ds+\alpha\bar{\sigma}^2 E_{\widehat{P}}\int_{0}^Te^{\beta s}|(Z^n_s-Z^{n-1}_s)\sigma(X_s))|^2ds.
\\
\leq& K\bar{\sigma}^2E_{\widehat{P}}\int_{0}^Te^{\beta s}|Y^n_s-Y^{n-1}_s|^2ds+\alpha\bar{\sigma}^2 \Lambda E_{\widehat{P}}\int_{0}^Te^{\beta s}|Z^n_s-Z^{n-1}_s|^2ds.
\end{split}
\end{equation*}
The strict ellipticity condition on matrix $a$ ensures that 
$$2\int_{0}^Te^{\beta s}\hat{Z}^{n+1}_s \cdot a(X_s)\cdot (\hat{Z}^{n+1}_s)^* ds\geq 2\lambda\int_{0}^Te^{\beta s}|\hat{Z}^{n+1}_s|^2 ds.$$
By similar arguments as previously, based on the Burkholder-Davies-Gundy inequality, we verify that ``martingales" terms in It\^o's formula vanish when taking expectation under $\widehat{P}$, so 
it follows that
\begin{equation*}
\begin{split} 
&(\beta-\frac{1}{\epsilon})E_{\widehat{P}}\int_{0}^Te^{\beta s}|\hat{Y}^{n+1}_s|^2ds+2\lambda E_{\widehat{P}}\int_{0}^Te^{\beta s}|\hat{Z}^{n+1}_s|^2 ds\\
\leq &\, K(\epsilon+\bar{\sigma}^2)E_{\widehat{P}}\int_{0}^Te^{\beta s}|Y^n_s-Y^{n-1}_s|^2ds+(K\epsilon+\alpha\Lambda\bar{\sigma}^2)E_{\widehat{P}}\int_{0}^Te^{\beta s}|Z^n_s-Z^{n-1}_s|^2ds.
\end{split}
\end{equation*}
We choose $\epsilon$ small enough and then $\beta$ such that 
$$\frac{K\epsilon +\alpha\Lambda\bar{\sigma}^2}{2\lambda}<1\ \mbox{and}\ \frac{\beta-\frac{1}{\epsilon}}{ 2\lambda}=\frac{K\epsilon+\bar{\sigma}^2}{K\epsilon+\alpha\Lambda\bar{\sigma}^2}.$$
Let $\delta=\frac{\beta-\frac{1}{\epsilon}}{ 2\lambda}$ and  set for $(Y,Z)\in  \hM^2_G  ([0,T], (\hF_t)_{t\in [0,T]})\times\left(\hM^2_G  ([0,T], (\hF_t)_{t\in [0,T]})\right)^d $
$$\|(Y,Z)\|_\delta:=\sup_{\widehat{P}\in\widehat{\cP}}\bigg(\delta E_{\widehat{P}}\int_{0}^T e^{\beta s}|{Y}_s|^2ds+ E_{\widehat{P}}\int_{0}^Te^{\beta s}|{Z}_s|^2 ds\bigg)^{1/2},$$
which obviously defines a norm equivalent to the product norm on \\$ \hM^2_G  ([0,T], (\hF_t)_{t\in [0,T]})\times\left(\hM^2_G  ([0,T], (\hF_t)_{t\in [0,T]})\right)^d$.
We have 
$$\forall n\in\N,\ \|(Y^{n+1},Z^{n+1})\|_\delta \leq \frac{K\epsilon +\alpha\Lambda\bar{\sigma}^2}{2\lambda}\| (Y^{n},Z^{n})\|_\delta,$$
we conclude thanks to the fixed point theorem that equation  \eqref{eqBDSDEs} admits a unique solution $(Y,Z)$  in $\hM^2_G ([0,T], (\hF_t)_{t\in [0,T]})\times\left(\hM^2_G  ([0,T], (\hF_t)_{t\in [0,T]})\right)^d$. Finally, it is easy to verify that $Y$ belongs to $\hS^2_G  ([0,T], (\hF_t)_{t\in [0,T]})$, this ends the proof of Theorem \ref{MainTheo}.

\section{Doubly stochastic representation}
\begin{proposition}\label{doublystochasticrepresentation}
Let $g$ be  in $\left( M^2_G([0,T],(\mathcal{F}_{t,T}^B)_{t\in [0,T]};L^2 (\R^d))\right)^l$, then the process 
$$\forall t\in [0,T], \ u_t =\int_t^T P_{s-t}g_s \cdot \Bf_s,$$
admits a doubly stochastic representation. More precisely,  for  $c$-quasi-all $\omega^2\in \Omega^B$, the next identity holds for $P^m$-almost all $\omega^1\in \Omega^W$ and all $t\in [0,T]$:
\begin{equation}\label{doublestochrepre}
\begin{split}u_t(\omega^2, X_t (\omega^1))=&\sum_{j=1}^l\int_t^T g_s^j (\omega^2,  X_s (\omega^1 ))\, \overleftarrow{dB}^j_s(\omega^2) \\&-\sum_{i=1}^d \int_t^T\partial_i u_s (\omega^2,X_s (\omega^1 ))\, dM^i_s(\omega^1).\end{split}\end{equation}
\end{proposition}
\begin{proof} Assume first that $g$ is of the form $g_s (\omega^2 ,x)=\varphi (\omega^2 ,s )g(x)$ with $g\in\cD (L)$, $\varphi$ uniformly bounded, adapted to the filtration  $(\mathcal{F}_{t,T}^B )_{t\in [0,T]}$ and such that for all $\omega^1 \in\Omega^W$, $s\rightarrow \varphi (s,\omega^2 )$ is $C^1$. Then 
$$\forall t\in [0,T], \ u_t =\int_t^T \varphi (s)P_{s-t}g \cdot \Bf_s =\int_t^T L(u_s) ds+\int_t^T g_s \cdot\Bf_s.$$
We set for all $0\leq t\leq s\leq T$ and $x\in \R^d$,  $\Delta^t_s(x):= \varphi (s) P_{s-t}g(x)$, then $\displaystyle\frac{\partial\Delta^t_s }{\partial s}=\varphi'(s)P_{s-t}g+\varphi (s) P_{s-t}(Lg)$ and by It\^o's formula, for any $t\in [0,T]$:
\begin{equation*}\begin{split}
\Delta_t^t\Bt_{T-t}=\Delta_{T-(T-t)}^t\Bt_{T-t}
&= -\int_0^{T-t}\frac{\partial\Delta^t_{T-s}}{\partial s}\Bt_s ds+\int_0^{T-t}\Delta^t_{T-s}d\Bt_s \\
&=-\int_t^{T}\frac{\partial\Delta^t_{s}}{\partial s}\Bt_{T-s} ds+\int_t^{T}\Delta^t_{s}\overleftarrow{dB}_s,
\end{split}\end{equation*}
hence
\begin{equation*}\begin{split}
u_t&=\Delta_t^t\Bt_{T-t}+\int_t^T \Bt_{T-s} P_{s-t}(\varphi' (s)g+\varphi (s)Lg)\, ds\\
&=\Delta_t^t\Bt_{T-t} +\int_t^T P_{s-t}\left( \Bt_{T-s} (\varphi' (s)g+\varphi (s)Lg)\right) ds.
\end{split}\end{equation*}
Let $\ut_t:=\int_t^T P_{s-t}\left(\Bt_{T-s} (\varphi'(s)g+\varphi (s)Lg)\right)  ds$,
as a consequence of Theorem 4.3 in \cite{BPS}, we have for quasi-all $\omega^2 \in\Omega^B$ and $P^m$-almost all $\omega^1\in\Omega^W$:
\begin{equation*}\begin{split}\ut_t (\omega^2 ,X_t (\omega^1))=&\,-\sum_{i=1}^d \int_t^T \partial_i \ut_s (\omega^2 ,X_s (\omega^1 ))\, dM^i_s (\omega^1)\\&+\int_t^T \left(\Bt_{T-s}(\omega^2 ) (\varphi' (s)g (X_s (\omega^1)+\varphi (s)Lg(X_s (\omega^1 )))\right)  ds,\end{split}\end{equation*}
for all $t\in [0,T]$.\\
On the other hand since $g \in\cD (L)$, it follows that 

$$g(X_t)=g(X_T)- \sum_{i=1}^d \int_t^T \partial_i g(X_s) \, dM^i_s-\int_t^T Lg(X_s )\,ds,$$
therefore,
\begin{equation*}\begin{split}
\Delta^t_t(X_t)=\varphi (t)g(X_t)=&\varphi(T)g(X_T)-\sum_{i=1}^d\int_t^T\varphi (s)\partial_i g(X_s)dM^i_s\\&-\int_t^T \varphi (s)Lg(X_s) ds-\int_t^T \varphi'(s)g(X_s)ds,\end{split}\end{equation*}
and so, by It\^o's formula \eqref{itoproductrule}, it holds
\begin{equation*}\begin{split} 
\Delta^t_t (X_t)\tilde{B}_{T-t}=&\, -\sum_{i=1}^d \int_t^T \varphi (s)\partial_i g(X_s)\tilde{B}_{T-s}\, dM^i_s-\int_t^T \varphi (s)Lg(X_s)\tilde{B}_{T-s}\, ds\\& -\int_t^T \varphi'(s)g(X_s )\tilde{B}_{T-s}\, ds+\int_t^T \varphi(s)g(X_s)\Bf_s\\
=&\, -\sum_{i=1}^d \int_t^T \partial_i \Delta_s^s (X_s) \tilde{B}_{T-s}\, dM^i_s-\int_t^T \varphi (s)Lg(X_s) \tilde{B}_{T-s}\, ds\\& -\int_t^T \varphi'(s)g(X_s ) \tilde{B}_{T-s}\, ds+\int_t^T g_s (X_s)\Bf_s.\\
\end{split}\end{equation*}
Now, since $u_t (X_t)=\ut_t (X_t)+\Delta^t_t (X_t) \tilde{B}_{T-t}$ and $\partial_i  u_t (X_t)=\partial_i \Delta_t^t (X_t)\tilde{B}_{T-t}+\partial_i\ut_t (X_t)$ we get the result in this case.\\
Let us turn out to the general case. Let $g$ be  in $\left( M^2_G([0,T],(\mathcal{F}_{t,T}^B )_{t\in [0,T]};L^2 (\R^d))\right)^l$, it is clear that it can be approximated in $\left( M^2_G([0,T],(\mathcal{F}_{t,T}^B )_{t\in [0,T]};L^2 (\R^d))\right)^l$ by a sequence $(g^n)_n $ of linear combinations of elementary processes as above. Let us define 
$$\forall n \in \N,\ \ u^n_t =\int_t^T P_{s-t}g^n_s \cdot \Bf_s \makebox{ and }\ u_t =\int_t^T P_{s-t}g_s \cdot \Bf_s.$$
For all $n\in\N$ we have the representation:
\begin{equation}
\begin{split}
u^n_t(\omega^2, X_t (\omega^1))=&\sum_{j=1}^l\int_t^T g_s^{n,j} (\omega^2,  X_s (\omega^1 ))\, \overleftarrow{dB}^j_s(\omega^2) \\&-\sum_{i=1}^d \int_t^T\partial_i u^n_s (\omega^2,X_s (\omega^1 ))\, dM^i_s(\omega^1).\end{split}\end{equation}
We put 
$$U^n_t=\sum_{j=1}^l\int_t^T g_s^{n,j} (\omega^2,  X_s (\omega^1 ))\, \overleftarrow{dB}^j_s(\omega^2)\makebox{ and } V^n_t=\sum_{i=1}^d \int_t^T\partial_i u^n_s (\omega^2,X_s (\omega^1 ))\, dM^i_s(\omega^1)$$
and 
$$U_t=\sum_{j=1}^l\int_t^T g_s^{j} (\omega^2,  X_s (\omega^1 ))\, \overleftarrow{dB}^j_s(\omega^2)\makebox{ and } V_t=\sum_{i=1}^d \int_t^T\partial_i u_s (\omega^2,X_s (\omega^1 ))\, dM^i_s(\omega^1).$$
Then, for all $P\in \cP$, thanks to the Doob's inequality and the fact that $P^m$ is the invariant measure of process $X$:
\begin{equation*}\begin{split}
    E_P\bigg[ E_{P^m}\Big[ \sup_{t\in [0,T]}|U^n_t -U_t|^2\Big]\bigg]&\leq 4\bar{\sigma}^2 E_P\bigg[ E_{P^m}\Big[ \int_0^T|g_t^n (X_t )-g_t (X_t)|^2dt \Big]\bigg]\\
    &= 4\bar{\sigma}^2 E_P\bigg[ \int_0^T\|g^n_t -g_t\|_{L^2 (\R^d)}^2dt \bigg]\underset{n\rightarrow +\infty}{\longrightarrow} 0.
\end{split}\end{equation*}
Similarly 
\begin{equation*}\begin{split}
    E_P\bigg[ E_{P^m}\Big[ \sup_{t\in [0,T]}|V^n_t -V_t|^2\Big]\bigg]&\leq 16\Lambda E_P\bigg[ E_{P^m}\Big[ \int_0^T|\nabla u^n (t,X_t )-\nabla u(t,X_t)|^2dt \Big]\bigg]\\
    &= 16\Lambda E_P\bigg[ \int_0^T\|\nabla u^n_t -\nabla u_t\|_{L^2 (\R^d;\R^d)}^2dt \bigg] \underset{n\rightarrow +\infty}{\longrightarrow} 0.
\end{split}\end{equation*}
It is now easy to conclude by density.
\end{proof}

\begin{corollary}
Let $u$ be a solution of GSPDE \eqref{GSPDE}. Then this solution admits a double stochastic representation as follows,
\begin{equation}\label{relationGSPDE-GBDSDE}\begin{split}
u(\omega^2,t,X_t(\omega^1))=&\,\Psi(X_T(\omega^1))\\+\int_t^T &f(\omega^2,s,X_s(\omega^1),u(\omega^2,s,X_s(\omega^1)),\nabla u\sigma(\omega^2,s,X_s(\omega^1)))ds\\+\int_t^T &g(\omega^2,s,X_s(\omega^1),u(\omega^2,s,X_s(\omega^1)),\nabla u\sigma(\omega^2,s,X_s(\omega^1)))\,\Bf_s(\omega^2)\\-\sum_{i=1}^d&\int_t^T\partial_iu(\omega^2,s,X_s(\omega^1))\,dM_s^i(\omega^1).
\end{split}\end{equation}
\end{corollary}

\begin{proof}
Taking into account formula \eqref{mildsolutionG} of the mild solution and the representation obtained in Proposition \ref{doublystochasticrepresentation}, it is enough to prove a double stochastic representation for a process of the following type
\begin{equation*}
v_t=P_{T-t}\Psi+\int_t^TP_{s-t}f_sds,
\end{equation*}
where $\Psi\in L^2(\mathbb{R}^d)$ and $f\in M^2_G([0,T],(\mathcal{F}_{t,T}^B)_{t\in[0,T]};L^2(\mathbb{R}^d))$. 
This process satisfies, for almost-all  $\omega^2\in\Omega^B$, the following deterministic equation 
\begin{equation*}
(\partial_t+L)v(\omega^2,t)+f(\omega^2,t)=0,\quad v(\omega^2,T)=\Psi.
\end{equation*}
Its representation then follows from \cite[Theorem 3.2]{Stoica}, more precisely:
\begin{equation}\label{stochrepredeter}
\begin{split}
v(\omega^2,t,X_t(\omega^1))=&\Psi(X_T(\omega^1))+\int_t^Tf(\omega^2,s,X_s(\omega^1))ds\\&-\sum_{i=1}^d\int_t^T\partial_iv(\omega^2,s,X_s(\omega^1))\,dM_s^i(\omega^1).
\end{split}\end{equation}
Then combining \eqref{doublestochrepre} and \eqref{stochrepredeter} we get the desired representation. 
\end{proof}

\end{document}